\documentclass[11pt,a4paper]{article}
\usepackage{tikz}
\usepackage{subfigure}
\usepackage[english]{babel}
\usepackage{graphicx}
\usepackage{extarrows}
\usepackage{lineno}
\usepackage{indentfirst, latexsym, bm, amssymb}
\usepackage{amsthm}
\usepackage{enumerate}
\usepackage{comment}
\numberwithin{figure}{section}

\begin{document}

\theoremstyle{plain}
\newtheorem{theorem}{Theorem}[section]
\newtheorem{lemma}{Lemma}[section]
\newtheorem{proposition}[theorem]{Proposition}
\newtheorem{corollary}[theorem]{Corollary}
\newtheorem{observation}{Observation}
\newtheorem{claim}{Claim}
\newtheorem{claim4}{Claim}[claim]

\theoremstyle{definition}
\newtheorem{definition}{Definition}[section]
\newtheorem{mydef}{Definition}[section]
\newtheorem{example}[theorem]{Example}
\newtheorem{conjecture}[theorem]{Conjecture}
\newtheorem{question}[theorem]{Question}

\newcommand{\remark}{\medskip\par\noindent {\bf Remark.~~}}
\newcommand{\pp}{{\it p.}}
\newcommand{\de}{\em}

\renewcommand{\theenumi}{Case \arabic{enumi}}
\renewcommand{\theenumii}{Subcase \arabic{enumi}.\arabic{enumii}}
\renewcommand{\labelenumii}{\theenumii.}

\title{Exact Tur\'{a}n numbers of two vertex-disjoint paths\footnote{Supported in part by Science and Technology Commission of Shanghai Municipality (No. 22DZ2229014) and National Natural Science Foundation of China (Nos. 12271169 and 12331014).}}
\author{Miao Dong\footnote{School of Mathematical Sciences, East China Normal University, Shanghai, 200241, China. Email:mdong@stu.ecnu.edu.cn},~Bo Ning\footnote{College of Computer Science, Nankai University, Tianjin, 300350, P.R. China.
Email: bo.ning@nankai.edu.cn (B. Ning). Partially supported by the NSFC (No. 12371350) and the Fundamental Research Funds for the Central
Universities (63243151).},~Long-Tu Yuan\footnote{School of Mathematical Sciences, Key Laboratory of MEA(Ministry of Education) \& Shanghai
Key Laboratory of PMMP, East China Normal University, Shanghai, 200241, China. Email: ltyuan@math.ecnu.edu.cn.}~ and~Xiao-Dong Zhang\footnote{School of Mathematical Sciences, MOE-LSC, SHL-MAC,
Shanghai Jiao Tong University, 800 Dongchuan Road, 200240, Shanghai, P.R.China. Email: xiaodong@sjtu.edu.cn (X.-D. Zhang). Partially supported by the National Natural Science Foundation of China (No.12371354)  and the Montenegrin-Chinese Science and Technology (No.4-3).} }
\date{}
\maketitle
\begin{abstract}
The Tur\'{a}n number of a graph $H$ is the maximum number of edges in any graph of order $n$ that does not contain $H$ as a subgraph.
In 1959, Erd\H os and Gallai obtained a sharp upper bound of Tur\'{a}n numbers for a path of arbitrary length. In 1975, Faudree
and Schelp, and independently in 1977, Kopylov determined the exact values of Tur\'an numbers
of paths with arbitrary length.
In this paper, we determine the Tur\'{a}n number of two vertex-disjoint paths of odd order at least 4.
Together with previous works, we determine the exact Tur\'an numbers of two vertex-disjoint paths completely.
This confirms the first $k=2$ case of a conjecture proposed by Yuan and Zhang in 2021,
which generalizes the Tur\'an number formula of paths due to Faudree-Schelp, and Kopylov in a
broader setting. Our main tools include a refinement of P\'{o}sa's rotation lemma, a stability
result of Kopylov's theorem on cycles, and a recent inequality on circumference, minimum degree,
and clique number of a 2-connected graph.
\end{abstract}


\vskip 0.5cm

\section{Introduction}

The \textit{Tur\'{a}n number} of a graph $H$, denote by $\mathrm{ex}(n,H)$,
is the maximum number of edges in a graph $G$ of order $n$ that does not contain $H$ as a subgraph.
Denote by $\mathrm{Ex}(n,H)$ the set of graphs on $n$ vertices with $\mathrm{ex}(n,H)$
edges that do not contain $H$ as a subgraph. We call a graph in $\mathrm{Ex}(n,H)$
an extremal graph for $H$.

As a major branch of graph theory, extremal graph theory aims to study the behaviors
of $\mathrm{ex}(n,H)$ and $\mathrm{Ex}(n,H)$. In 1941, Tur\'an \cite{Turan1941}
determined that $\mathrm{Ex}(n,K_{r+1})=T_{n,r}$, where $K_{r+1}$ denotes
the complete graph of order $r+1\geq 3$, and $T_{n,r}$ denotes the balanced
$r$-partite Tur\'an graph on $n$ vertices. The celebrated Erd\H{o}s-Stone-Simonovits
Theorem said that for any graph $H$ with chromatic number $\chi(H)=k+1\geq 2$,
we have $\mathrm{ex}(n,H)=(1-1/k+o(1))(n^2/{2})$,
which asymptotically determined Tur\'an numbers of all graphs with chromatic number at least 3.
When $\chi(H)=2$, the Erd\H{o}s-Stone-Simonovits Theorem implies that
$\mathrm{ex}(n,H)=o(n^2)$ and does not provide us much information.
Until now, the study of Tur\'an number of bipartite graphs is a central
topic in extremal graph theory, and our knowledge about
it is rather limited. For a wonderful survey on this topic, we refer
to \cite{Furedi2013}.

Among bipartite graphs, the complete picture of the paths is known to us.
Denote by $P_k$ the path of order $k$. Moreover, if $k$ is odd then $P_k$
is called an odd path; otherwise an even path.
Dating back to 1959, Erd\H{o}s and Gallai \cite{Erdos-Gallai-1959} proved a classic result
on Tur\'an number of $P_k$ as follows.

\begin{theorem}[\cite{Erdos-Gallai-1959}]\label{THM: path}
Let $n\geq k\geq 3$ and $t$ be integers. Then
\begin{equation*}
          \mathrm{ex}(n,P_k)\leq \frac{(k-2)n}{2},
\end{equation*}
where equality holds if and only if $n=t(k-1)$.
\end{theorem}

Note that Erd\H{o}s-Gallai Theorem determines the exact Tur\'an numbers
of paths when $n=t(k-1)$. In 1975, Faudree and Schelp \cite{Faudree1975},
and independently, Kopylov \cite{Kopylov1977} in 1977 determined all the
values of $\mathrm{ex}(n,P_k)$.

\begin{theorem}[\cite{Faudree1975,Kopylov1977}]\label{extrem}
Let $n\geq k\geq 3$. Then $$\mathrm{ex}(n,P_k)=s{k-1\choose 2}+{r\choose 2},$$
where $n=s(k-1)+r$ and $0\leq r\leq k-2$.
\end{theorem}

It is very natural to consider the Tur\'an number of disjoint paths.
Let $G$ and $H$ be two graphs, we denote by $G\cup H$ the \emph{disjoint union} of $G$ and $H$.
For a positive integer $t\geq 2$, denote by $tG$ the disjoint union of $t$ copies
of $G$ and $M_t$ a matching of size $t$. Obviously, $M_t$ equals to $tP_2$.
The first famous result about matchings (i.e., disjoint $P_2$'s) is also due to
Erd\H{o}s and Gallai \cite{Erdos-Gallai-1959}.
\begin{theorem}[\cite{Erdos-Gallai-1959}]
Let $n,r$ be two positive integers with $n\geq 2r$ and $r\geq 1$. Then
$$
\mathrm{ex}(n,M_t)=\max\left\{\binom{2t-1}{2},\binom{t-1}{2}+(t-1)(n-t+1)\right\}.
$$
\end{theorem}

In \cite{Liu2013}, Lidick\'{y}, Liu, and Palmer determined the Tur\'an number of disjoint paths when $n$ is sufficiently large.
The extremal graphs for a given graph when $n$ is small are different for those when $n$ is large in most cases (for example see \cite{erdHos1965,Furedi20152}).
As a considerable extension of Erd\H{o}s-Gallai Theorem and Lidick\'{y}-Liu-Palmer Theorem, the third and fourth
authors \cite{Yuan-Zhang2021} proposed the following conjecture on the exact Tur\'an number of disjoint paths (for all values of $n$).
To understand the conjecture, we first need to introduce the following symbols.

\begin{definition}
Let $n\geq m\geq \ell\ge 3$ be given three positive integers. Then $n$ can be written as $n=(m-1)+t(\ell-1)+r$, where $t\geq 0$ and $ 0\leq r < \ell-1.$  Denote by
 $$c(n,m,\ell)\equiv {m-1 \choose 2}+t{\ell-1 \choose 2}+{r \choose 2}$$
  and
  $$c(n,m)\equiv {\lfloor m/2\rfloor-1 \choose 2}+\lfloor (m-2)/2 \rfloor\left(n-\lfloor m/2 \rfloor+1\right).$$ Moreover, if  $n\leq m-1$ then denote by $$c(n,m,\ell)\equiv{n \choose 2}.$$
\end{definition}

\begin{conjecture}[\cite{Yuan-Zhang2021}]\label{Conj:Yuan-Zhang}
Let $k_1\geq k_2\geq\ldots\geq k_{m}\geq 3$ and $k_1>3$. If $F_m=
P_{k_1}\cup P_{k_2}\cup \ldots \cup P_{k_m}$, then
\begin{eqnarray*}
\mbox{ex}(n, F_m) = \max\left\{c(n,k_1,k_1),c(n,k_1+k_2,k_2),\ldots,c\left(n,\sum_{i=1}^{m}k_{i},k_{m}\right),c\left(n,\sum_{i=1}^{m}\lfloor k_{i}/2 \rfloor\right)+c\right\},
\end{eqnarray*}
where $c=1$ if all of $k_1,k_2,\ldots,k_m$ are odd; and $c=0$ otherwise.
\end{conjecture}

The third and fourth authors confirmed the above conjecture for the case
at most $k_1,k_2,\ldots,k_m$ is odd (see Theorem 1.7 in \cite{Yuan-Zhang2021}).

One may find that the above conjecture cannot cover previous results, which
permit each odd path to be of order 3 and is also very interesting.
For this topic, Gorgoal (2011, \cite{gorgol}) determined $\mathrm{ex}(n,3P_3)$ and posed a conjecture
on $\mathrm{ex}(n,kP_3)$ for sufficiently large $n$. Her conjecture was solved
in a stronger form by Bushaw and Kettle (2011, \cite{Bushwa2011}),
in which they determined $\mathrm{ex}(n,kP_3)$ for $n\geq 7k$. Later, the exact
values of $\mathrm{ex}(n,kP_3)$ for all the values of $k,n$ were determined by Campos
and Lopes (2018, \cite{Campos-Lopes2018}), and independently by Yuan and Zhang
(2017, \cite{Yuan-Zhang2017}).
Moreover, Yuan and Zhang \cite{Yuan-Zhang2017} obtained all extremal graphs in $\mathrm{EX}(n,kP_3)$.
In the other direction, Bielak and Kieliszek \cite{Bielak2014,Bielak2016} (2014, 2016)
determined $\mathrm{ex}(n,3P_4)$ and $\mathrm{ex}(n,2P_5)$, respectively.
Deng, Hou, and Zeng (2024, \cite{Deng-Hou-Zeng2024}) determined $\mathrm{ex}(n,3P_7)$
and $\mathrm{ex}(n,2P_3\cup P_{2\ell+1})$.
Very recently, Lai, Ma, and Yan (2026, \cite{LMY2026}) extended Deng et al.'s result and
determined $\mathrm{ex}(n,kP_3\cup rP_2)$.
All of these results partially support Yuan-Zhang's conjecture. (The fact that Lai et
al.'s result partially supports the Yuan-Zhang conjecture is not very obvious.
We refer the reader to the appendix part of \cite{LMY2026}
for the reason.)

The main purpose of this paper is to confirm the first case $k=2$ of
Conjecture \ref{Conj:Yuan-Zhang} completely. To state a previous related
result (also related to the case $k=2$), we need to introduce two definitions.

\begin{mydef}\label{fig:3}
For integers $n\geq k \geq 2a$, let $H(n,k,a)$ be the $n$-vertex graph whose
vertex set is partitioned into three sets $A,B,C$ such that $|B| = a$,
$|C| = n-k+a$, $|A| = k-2a$ and whose edge set consists of all edges between
$B$ and $C$ together with all edges in $A\cup B$ (see Figure \ref{fig:3},
the subgraphs induced by $A$ and $B$ are complete graphs and the subgraph
induced by $C$ contains no edge). Let
\begin{equation}\label{ema}
 h(n, k, a):=e(H(n,k,a))={k-a\choose 2}+(n-k+a)a.
 \end{equation}
\end{mydef}

\remark Note that each path/cycle in $H(n,k,a)$ cannot contain consecutive vertices in $C$. One may check that the longest
     path in $H(n, k, a)$ contains $k$ vertices and the longest cycle in $H(n, k, a)$ contains $k- 1$ vertices.

\begin{mydef}
          For integers $n\geq s+t-1$ and $k_2\geq 2$, let $\mathcal{C}(n,s,t)$ be the set of $n$-vertex graphs which are partitioned
          into $K_{s+t-1}$ and a graph in $\mathrm{Ex}(n-s-t+1,P_{2k_2+1})$.
 \end{mydef}

Yuan and Zhang \cite{Yuan-Zhang2021} determined the maximum number of edges in a
graph $G$ of order $n$ which does not contain $P_{2\ell+1}\cup P_3$($\ell\geq 1$) as a subgraph.

\begin{theorem}[\cite{Yuan-Zhang2021}]\label{Y}
     Let $n>2\ell+4$. Then
     \begin{equation*}
          \mathrm{ex}(n,P_{2\ell+1}\cup P_3)=\max\{c(n,2\ell+1,3),\mathrm{ex}(n,P_{2\ell+1})  , h(n,2\ell,\ell-1)  \}.
     \end{equation*}
\end{theorem}
Our main theorem in this paper is as follows.

\begin{theorem}\label{turan two odd path}
Let $G$ be a graph on $n$ vertices and $k_1\geq k_2>3$ be odd numbers. Then
$$
\mathrm{ex}(n,P_{k_1}\cup P_{k_2})=\max\left\{c(n,k_1,k_2),\mathrm{ex}(n,P_{k_1}),h(n,k_1+k_2-2,\lfloor k_1/2\rfloor+\lfloor k_2/2\rfloor-1)\right\}.
$$
\end{theorem}

Together with Theorem \ref{Y} above and Theorem 1.7 in \cite{Yuan-Zhang2021}, we
confirm Conjecture \ref{Conj:Yuan-Zhang} for the base $k=2$ case completely.

We point out that the main hardness of our theorem lies on the assumption that the order
of a graph can be small, instead of that the assumption
$n$ is sufficiently large. One example is that, Erd\H{o}s \cite{Erdos1962},
Moon \cite{Moon1968}, and Simonovits \cite{Simonovits1968} proved that
$\mathrm{EX}(n,tK_{p+1})=\{K_{t-1}\vee T_{n-t+1}(p)\}$
for $n$ sufficiently large. But only very recently, after finishing three papers by
several people, the complete pictures of $\mathrm{ex}(n,K_p\cup K_q)$ and
$\mathrm{EX}(n,K_p\cup K_q)$ are known. In fact, $\mathrm{ex}(n,2K_{p+1})$ was
completely determined by Chen, Lu, and Yuan
\cite{Chen-Lu-Yuan2022} in 2021. Hu \cite{Hu2024} determined $\mathrm{ex}(n, K_p\cup K_q)$
for all $p\geq q+1$, and $\mathrm{EX}(n,K_p\cup K_q)$ where $p\geq q+1$ was
determined by Luo \cite{Luo2025} in 2025.
Moreover, our proof needs three new tools on paths, and we highlight and
postpone them and the proofs to Section \ref{Sec:2}.

\medskip

\noindent
{\bf Our approach.} Let $G$ be a  connected graph which does not contain
a path on $k$ vertices. The graph $G^\prime$ obtained from $G$ by adding
a new vertex and joining it to all vertices of $G$ is 2-connected without
of cycles of length at least $k+1$. Note that after deleting any vertex
of a cycle, the resulting graph is a path. Hence one can easily determine
the maximum number of edges in an $n$-vertex connected graph which does
not contain a path on $k$ vertices from the maximum number of edges in an
$(n+1)$-vertex 2-connected graph with circumference at most $k$. With these
ideas in mind, we will first find a family of graphs (see
Section~\ref{Sec:2}) such that after deleting
any vertex of them, each of the resulting graphs contains two odd paths
with given lengths (Lemmas~\ref{property for H} and~\ref{H(2k+1)}),
which is the most hardest part of our proof. Next, we combine P\'{o}sa
Rotation technique and the ideas from Kopylov to get an upper bound of the maximum
number of edges in an $n$-vertex 2-connected graph which contains no graphs
in Section~\ref{Sec:3} (Lemma \ref{Lemma:Stablity-Kopylov2}).
Hence we can get the upper bound for the connected Tur\'{a}n number of two
vertex-disjoint odd paths as we discuss at the beginning of this paragraph
(Theorem~\ref{3.2}). Finally, we can derive our main result (Theorem~\ref{turan two odd path})
by considering each component of the extremal graph for two disjoint odd paths.

\medskip

\noindent
{\bf Organisation.}
This paper is organised as follows. Section \ref{Sec:2} includes three new tools on paths,
i.e., an inequality on circumference, minimum degree, and clique number of a 2-connected graph (Lemma \ref{Lemma:Yuan}),
a stability result on Kopylov's theorem on cycles (Lemma \ref{Lemma:Stablity-Kopylov2}),
and a refinement of Pos\'{a}'s lemma
(Lemma \ref{extend posa lemma}). We also prove Lemma \ref{extend posa lemma} in this section.
In Section \ref{Sec:3}, we prove several new lemmas and depict many extremal graphs.
In Sections \ref{Sec:4} and \ref{Sec:5}, we prove Theorem \ref{turan two odd path}
(assuming Lemma \ref{Lemma:Stablity-Kopylov2}), and later prove
Lemma \ref{Lemma:Stablity-Kopylov2}.

\section{Three new tools on paths}\label{Sec:2}
In this section, we collect three new tools on paths.

As preliminaries, we introduce
some notation. All the graphs considered are finite, undirected, and simple (no loops or multiple edges).
Let $G=(V(G),E(G))$ be a simple graph, where $V (G)$ is the vertex set and $E(G)$ is the edge set with size $e(G)$.
If $S\subset V (G)$, the subgraph of $G$ induced by $S$ is denoted by $G[S]$.
Denote by $M_t$ a matching of size $t$, and $K_{m,n}$ the complete bipartite graph with two
bipartition with $m$ and $n$ vertices, receptively.
Moreover, we call a path $(x,y)$-path if it starts at $x$ and ends in $y$.
The circumference $c(G)$ of a graph $G$ is the length of a longest cycle in $G$.
Let $\delta(G)$ be the minimum degree of $G$. The clique number, denoted by $\omega(G)$,
is the maximum order of a clique in $G$.
For $n = k - t - 2 + \ell(t - 1) + 2$, the graph $Z(n, k, t)$ denotes the $n$-vertex graph
obtained from vertex-disjoint union of a clique $K_{k-t-2}$ and $\ell \geq 2$ copies of $K_{t-1}$
by adding two new vertices and joining them completely to all other vertices.

We first need the following lemma concerning circumference, minimum degree, and
clique number proved by  the third author in \cite{Yuan2024}.
\begin{lemma}[\cite{Yuan2024}]\label{Lemma:Yuan}
          Let $G$ be a 2-connected $n$-vertex graph. Then $$c(G)\geq \min\{n, \omega(G)+\delta(G)\},$$ unless $G$ is either  $H(n,\omega(G)+\delta(G),\delta(G))$ or $Z(n,\omega(G)+\delta(G),\delta(G))$.
\end{lemma}

In \cite{W76}, Woodall proposed the following conjecture: If $k\geq 3$, $2\leq \delta(G) \leq  (k-1)/2$
and $n\geq k$, and $G$ is a 2-connected $n$-vertex graph with more than
$\max\{h(n, k,\delta(G)),h(n,k,\lfloor (k-1)/2\rfloor)\}$ edges,
then $G$ contains a cycle of length at least $k$.
Ma and the second author \cite{MN2020} proved a stability result of Woodall's conjecture, which extended the previous
work of F\"{u}redi et al. \cite{FKV16}. In fact, Lemma \ref{Lemma:Yuan} is closely related to Lemma 4.4 in \cite{MN2020}.

Second, we need a stability result of Kopylov's theorem.
Recall that in 1977, Kopylov \cite{Kopylov1977} proved the following theorem.
\begin{theorem}[\cite{Kopylov1977}]\label{Kopylov}
    Let $k\geq 5$ and $G$ be a 2-connected graph. If
     \begin{equation*}
     e(G)>\max\{h(n,k,\lfloor k/2\rfloor),h(n,k,2)\},
     \end{equation*}
     then $G$ contains a cycle of length at least $k$.
     \end{theorem}

The following theorem whose proof is postpone to Section \ref{Sec:5} is a stability
result of Theorem~\ref{Kopylov} concerning circumference, that is,
if we forbidden more subgraphs of a graph, then the extremal number of edges of such graphs should be smaller.
We begin with the following definitions.
\begin{mydef}
     For a positive integer $k$, denote by $\mathcal{C}_{2k+1}$ the set of all cycles with length at least $2k+1$.
\end{mydef}

\begin{mydef}\label{fg}
     For an integer $k\geq 4$, let $H_k(P_3)$ and $H_k(M_2)$ be the graphs obtained from $K_{k-1,k+2}$ by embedding a copy of $P_3$ and a  copy of $M_2$
 into the larger part of $K_{k-1,k+2}$ respectively.
   \end{mydef}

\begin{mydef}\label{fig:1}
For integers $k\geq 4$ and $s\leq k-1$,
let $P_{2k+1}=x_1\ldots x_{2k+1}$ be a path on $2k+1$ vertices.
Let $A=\{x_1,\ldots, x_{s}\}$, $B=\{x_{2k+1},\ldots,x_{2k-s+2}\}$, $C=\{x_{s+1},x_{s+3},\ldots,x_{2k-s+1}\}$, and
$D=\{x_{s+2},$ $x_{s+4},\ldots,x_{2k-s}\}$.
Denote by $H_k(s)$ the graph obtained from $P_{2k+1}=x_1\ldots x_{2k+1}$ by adding all edges
inside $A$ and $B$, and all edges between $A\cup B$ and $C$.
\end{mydef}
Note that the graph obtained from $H_k(2)$ by adding all edges between $C$
and $D$ is $H_k(M_2)$ (see Figure \ref{fig:1}).\\

     \begin{center}
     \tikzset{every picture/.style={line width=0.75pt}} 
     \begin{tikzpicture}[x=0.75pt,y=0.75pt,yscale=-1,xscale=1]
     \draw    (31.67,85.08) -- (67.75,85.08) ;
     \draw [shift={(67.75,85.08)}, rotate = 0] [color={rgb, 255:red, 0; green, 0; blue, 0 }  ][fill={rgb, 255:red, 0; green, 0; blue, 0 }  ][line width=0.75]      (0, 0) circle [x radius= 1.34, y radius= 1.34]   ;
     \draw [shift={(31.67,85.08)}, rotate = 0] [color={rgb, 255:red, 0; green, 0; blue, 0 }  ][fill={rgb, 255:red, 0; green, 0; blue, 0 }  ][line width=0.75]      (0, 0) circle [x radius= 1.34, y radius= 1.34]   ;
     \draw    (188.02,85.08) -- (224.1,85.08) ;
     \draw [shift={(224.1,85.08)}, rotate = 0] [color={rgb, 255:red, 0; green, 0; blue, 0 }  ][fill={rgb, 255:red, 0; green, 0; blue, 0 }  ][line width=0.75]      (0, 0) circle [x radius= 1.34, y radius= 1.34]   ;
     \draw [shift={(188.02,85.08)}, rotate = 0] [color={rgb, 255:red, 0; green, 0; blue, 0 }  ][fill={rgb, 255:red, 0; green, 0; blue, 0 }  ][line width=0.75]      (0, 0) circle [x radius= 1.34, y radius= 1.34]   ;
     \draw    (67.75,85.08) -- (79.78,18.55) ;
     \draw [shift={(79.78,18.55)}, rotate = 280.25] [color={rgb, 255:red, 0; green, 0; blue, 0 }  ][fill={rgb, 255:red, 0; green, 0; blue, 0 }  ][line width=0.75]      (0, 0) circle [x radius= 1.34, y radius= 1.34]   ;
     \draw [shift={(67.75,85.08)}, rotate = 280.25] [color={rgb, 255:red, 0; green, 0; blue, 0 }  ][fill={rgb, 255:red, 0; green, 0; blue, 0 }  ][line width=0.75]      (0, 0) circle [x radius= 1.34, y radius= 1.34]   ;
     \draw    (103.83,118.34) -- (79.78,18.55) ;
     \draw [shift={(79.78,18.55)}, rotate = 256.45] [color={rgb, 255:red, 0; green, 0; blue, 0 }  ][fill={rgb, 255:red, 0; green, 0; blue, 0 }  ][line width=0.75]      (0, 0) circle [x radius= 1.34, y radius= 1.34]   ;
     \draw [shift={(103.83,118.34)}, rotate = 256.45] [color={rgb, 255:red, 0; green, 0; blue, 0 }  ][fill={rgb, 255:red, 0; green, 0; blue, 0 }  ][line width=0.75]      (0, 0) circle [x radius= 1.34, y radius= 1.34]   ;
     \draw    (151.94,118.34) -- (127.88,18.55) ;
     \draw [shift={(127.88,18.55)}, rotate = 256.45] [color={rgb, 255:red, 0; green, 0; blue, 0 }  ][fill={rgb, 255:red, 0; green, 0; blue, 0 }  ][line width=0.75]      (0, 0) circle [x radius= 1.34, y radius= 1.34]   ;
     \draw [shift={(151.94,118.34)}, rotate = 256.45] [color={rgb, 255:red, 0; green, 0; blue, 0 }  ][fill={rgb, 255:red, 0; green, 0; blue, 0 }  ][line width=0.75]      (0, 0) circle [x radius= 1.34, y radius= 1.34]   ;
     \draw    (188.02,85.08) -- (175.99,18.55) ;
     \draw [shift={(175.99,18.55)}, rotate = 259.75] [color={rgb, 255:red, 0; green, 0; blue, 0 }  ][fill={rgb, 255:red, 0; green, 0; blue, 0 }  ][line width=0.75]      (0, 0) circle [x radius= 1.34, y radius= 1.34]   ;
     \draw [shift={(188.02,85.08)}, rotate = 259.75] [color={rgb, 255:red, 0; green, 0; blue, 0 }  ][fill={rgb, 255:red, 0; green, 0; blue, 0 }  ][line width=0.75]      (0, 0) circle [x radius= 1.34, y radius= 1.34]   ;
     \draw    (103.83,118.34) -- (127.88,18.55) ;
     \draw [shift={(127.88,18.55)}, rotate = 283.55] [color={rgb, 255:red, 0; green, 0; blue, 0 }  ][fill={rgb, 255:red, 0; green, 0; blue, 0 }  ][line width=0.75]      (0, 0) circle [x radius= 1.34, y radius= 1.34]   ;
     \draw [shift={(103.83,118.34)}, rotate = 283.55] [color={rgb, 255:red, 0; green, 0; blue, 0 }  ][fill={rgb, 255:red, 0; green, 0; blue, 0 }  ][line width=0.75]      (0, 0) circle [x radius= 1.34, y radius= 1.34]   ;
     \draw    (151.94,118.34) -- (175.99,18.55) ;
     \draw [shift={(175.99,18.55)}, rotate = 283.55] [color={rgb, 255:red, 0; green, 0; blue, 0 }  ][fill={rgb, 255:red, 0; green, 0; blue, 0 }  ][line width=0.75]      (0, 0) circle [x radius= 1.34, y radius= 1.34]   ;
     \draw [shift={(151.94,118.34)}, rotate = 283.55] [color={rgb, 255:red, 0; green, 0; blue, 0 }  ][fill={rgb, 255:red, 0; green, 0; blue, 0 }  ][line width=0.75]      (0, 0) circle [x radius= 1.34, y radius= 1.34]   ;
     \draw    (79.78,18.55) -- (224.1,85.08) ;
     \draw    (284.23,85.08) -- (320.31,85.08) ;
     \draw [shift={(320.31,85.08)}, rotate = 0] [color={rgb, 255:red, 0; green, 0; blue, 0 }  ][fill={rgb, 255:red, 0; green, 0; blue, 0 }  ][line width=0.75]      (0, 0) circle [x radius= 1.34, y radius= 1.34]   ;
     \draw [shift={(284.23,85.08)}, rotate = 0] [color={rgb, 255:red, 0; green, 0; blue, 0 }  ][fill={rgb, 255:red, 0; green, 0; blue, 0 }  ][line width=0.75]      (0, 0) circle [x radius= 1.34, y radius= 1.34]   ;
     \draw    (440.58,85.08) -- (476.66,85.08) ;
     \draw [shift={(476.66,85.08)}, rotate = 0] [color={rgb, 255:red, 0; green, 0; blue, 0 }  ][fill={rgb, 255:red, 0; green, 0; blue, 0 }  ][line width=0.75]      (0, 0) circle [x radius= 1.34, y radius= 1.34]   ;
     \draw [shift={(440.58,85.08)}, rotate = 0] [color={rgb, 255:red, 0; green, 0; blue, 0 }  ][fill={rgb, 255:red, 0; green, 0; blue, 0 }  ][line width=0.75]      (0, 0) circle [x radius= 1.34, y radius= 1.34]   ;
     \draw    (320.31,85.08) -- (332.34,18.55) ;
     \draw [shift={(332.34,18.55)}, rotate = 280.25] [color={rgb, 255:red, 0; green, 0; blue, 0 }  ][fill={rgb, 255:red, 0; green, 0; blue, 0 }  ][line width=0.75]      (0, 0) circle [x radius= 1.34, y radius= 1.34]   ;
     \draw [shift={(320.31,85.08)}, rotate = 280.25] [color={rgb, 255:red, 0; green, 0; blue, 0 }  ][fill={rgb, 255:red, 0; green, 0; blue, 0 }  ][line width=0.75]      (0, 0) circle [x radius= 1.34, y radius= 1.34]   ;
     \draw    (356.39,118.34) -- (332.34,18.55) ;
     \draw [shift={(332.34,18.55)}, rotate = 256.45] [color={rgb, 255:red, 0; green, 0; blue, 0 }  ][fill={rgb, 255:red, 0; green, 0; blue, 0 }  ][line width=0.75]      (0, 0) circle [x radius= 1.34, y radius= 1.34]   ;
     \draw [shift={(356.39,118.34)}, rotate = 256.45] [color={rgb, 255:red, 0; green, 0; blue, 0 }  ][fill={rgb, 255:red, 0; green, 0; blue, 0 }  ][line width=0.75]      (0, 0) circle [x radius= 1.34, y radius= 1.34]   ;
     \draw    (404.5,118.34) -- (380.44,18.55) ;
     \draw [shift={(380.44,18.55)}, rotate = 256.45] [color={rgb, 255:red, 0; green, 0; blue, 0 }  ][fill={rgb, 255:red, 0; green, 0; blue, 0 }  ][line width=0.75]      (0, 0) circle [x radius= 1.34, y radius= 1.34]   ;
     \draw [shift={(404.5,118.34)}, rotate = 256.45] [color={rgb, 255:red, 0; green, 0; blue, 0 }  ][fill={rgb, 255:red, 0; green, 0; blue, 0 }  ][line width=0.75]      (0, 0) circle [x radius= 1.34, y radius= 1.34]   ;
     \draw    (440.58,85.08) -- (428.55,18.55) ;
     \draw [shift={(428.55,18.55)}, rotate = 259.75] [color={rgb, 255:red, 0; green, 0; blue, 0 }  ][fill={rgb, 255:red, 0; green, 0; blue, 0 }  ][line width=0.75]      (0, 0) circle [x radius= 1.34, y radius= 1.34]   ;
     \draw [shift={(440.58,85.08)}, rotate = 259.75] [color={rgb, 255:red, 0; green, 0; blue, 0 }  ][fill={rgb, 255:red, 0; green, 0; blue, 0 }  ][line width=0.75]      (0, 0) circle [x radius= 1.34, y radius= 1.34]   ;
     \draw    (356.39,118.34) -- (380.44,18.55) ;
     \draw [shift={(380.44,18.55)}, rotate = 283.55] [color={rgb, 255:red, 0; green, 0; blue, 0 }  ][fill={rgb, 255:red, 0; green, 0; blue, 0 }  ][line width=0.75]      (0, 0) circle [x radius= 1.34, y radius= 1.34]   ;
     \draw [shift={(356.39,118.34)}, rotate = 283.55] [color={rgb, 255:red, 0; green, 0; blue, 0 }  ][fill={rgb, 255:red, 0; green, 0; blue, 0 }  ][line width=0.75]      (0, 0) circle [x radius= 1.34, y radius= 1.34]   ;
     \draw    (404.5,118.34) -- (428.55,18.55) ;
     \draw [shift={(428.55,18.55)}, rotate = 283.55] [color={rgb, 255:red, 0; green, 0; blue, 0 }  ][fill={rgb, 255:red, 0; green, 0; blue, 0 }  ][line width=0.75]      (0, 0) circle [x radius= 1.34, y radius= 1.34]   ;
     \draw [shift={(404.5,118.34)}, rotate = 283.55] [color={rgb, 255:red, 0; green, 0; blue, 0 }  ][fill={rgb, 255:red, 0; green, 0; blue, 0 }  ][line width=0.75]      (0, 0) circle [x radius= 1.34, y radius= 1.34]   ;
     \draw    (332.34,18.55) -- (476.66,85.08) ;
     \draw [color={rgb, 255:red, 246; green, 16; blue, 44 }  ,draw opacity=1 ]   (332.34,18.55) -- (404.5,118.34) ;
     \draw [color={rgb, 255:red, 246; green, 16; blue, 44 }  ,draw opacity=1 ]   (356.39,118.34) -- (428.55,18.55) ;
     \draw    (31.67,85.08) -- (79.78,18.55) ;
     \draw    (31.67,85.08) -- (127.88,18.55) ;
     \draw    (31.67,85.08) -- (175.99,18.55) ;
     \draw    (67.75,85.08) -- (127.88,18.55) ;
     \draw    (67.75,85.08) -- (175.99,18.55) ;
     \draw    (188.02,85.08) -- (127.88,18.55) ;
     \draw    (224.1,85.08) -- (175.99,18.55) ;
     \draw    (188.02,85.08) -- (79.78,18.55) ;
     \draw    (224.1,85.08) -- (127.88,18.55) ;
     \draw    (284.23,85.08) -- (332.34,18.55) ;
     \draw    (284.23,85.08) -- (380.44,18.55) ;
     \draw    (284.23,85.08) -- (428.55,18.55) ;
     \draw    (320.31,85.08) -- (380.44,18.55) ;
     \draw    (320.31,85.08) -- (428.55,18.55) ;
     \draw    (476.66,85.08) -- (428.55,18.55) ;
     \draw    (440.58,85.08) -- (332.34,18.55) ;
     \draw    (476.66,85.08) -- (380.44,18.55) ;
     \draw    (440.58,85.08) -- (380.44,18.55) ;
     \draw   (13.71,85.08) .. controls (13.71,76.41) and (29.83,69.39) .. (49.71,69.39) .. controls (69.59,69.39) and (85.7,76.41) .. (85.7,85.08) .. controls (85.7,93.74) and (69.59,100.77) .. (49.71,100.77) .. controls (29.83,100.77) and (13.71,93.74) .. (13.71,85.08) -- cycle ;
     \draw   (87.83,118.34) .. controls (87.83,109.73) and (105.76,102.74) .. (127.88,102.74) .. controls (150,102.74) and (167.94,109.73) .. (167.94,118.34) .. controls (167.94,126.96) and (150,133.95) .. (127.88,133.95) .. controls (105.76,133.95) and (87.83,126.96) .. (87.83,118.34) -- cycle ;
     \draw   (340.39,118.34) .. controls (340.39,109.73) and (358.32,102.74) .. (380.44,102.74) .. controls (402.56,102.74) and (420.5,109.73) .. (420.5,118.34) .. controls (420.5,126.96) and (402.56,133.95) .. (380.44,133.95) .. controls (358.32,133.95) and (340.39,126.96) .. (340.39,118.34) -- cycle ;
     \draw   (58.35,18.55) .. controls (58.35,10.68) and (89.48,4.3) .. (127.88,4.3) .. controls (166.28,4.3) and (197.41,10.68) .. (197.41,18.55) .. controls (197.41,26.42) and (166.28,32.8) .. (127.88,32.8) .. controls (89.48,32.8) and (58.35,26.42) .. (58.35,18.55) -- cycle ;
     \draw   (310.91,18.55) .. controls (310.91,10.68) and (342.04,4.3) .. (380.44,4.3) .. controls (418.84,4.3) and (449.97,10.68) .. (449.97,18.55) .. controls (449.97,26.42) and (418.84,32.8) .. (380.44,32.8) .. controls (342.04,32.8) and (310.91,26.42) .. (310.91,18.55) -- cycle ;
     \draw   (170.06,85.08) .. controls (170.06,76.41) and (186.18,69.39) .. (206.06,69.39) .. controls (225.93,69.39) and (242.05,76.41) .. (242.05,85.08) .. controls (242.05,93.74) and (225.93,100.77) .. (206.06,100.77) .. controls (186.18,100.77) and (170.06,93.74) .. (170.06,85.08) -- cycle ;
     \draw   (266.27,85.08) .. controls (266.27,76.41) and (282.39,69.39) .. (302.27,69.39) .. controls (322.15,69.39) and (338.26,76.41) .. (338.26,85.08) .. controls (338.26,93.74) and (322.15,100.77) .. (302.27,100.77) .. controls (282.39,100.77) and (266.27,93.74) .. (266.27,85.08) -- cycle ;
     \draw   (422.62,85.08) .. controls (422.62,76.41) and (438.74,69.39) .. (458.62,69.39) .. controls (478.49,69.39) and (494.61,76.41) .. (494.61,85.08) .. controls (494.61,93.74) and (478.49,100.77) .. (458.62,100.77) .. controls (438.74,100.77) and (422.62,93.74) .. (422.62,85.08) -- cycle ;

     \draw (127.88,137.23) node [anchor=north] [inner sep=0.75pt]  [font=\small] [align=left] {$H_4(2)$};
     \draw (403.89,137.23) node [anchor=north east] [inner sep=0.75pt]  [font=\small] [align=left] {$H_4(M_2)$};
     \draw (254.16,145.79) node [anchor=north] [inner sep=0.75pt]   [align=left] {Figure \ref{fig:1}.};
     \draw (36.67,86.08) node [anchor=north east] [inner sep=0.75pt]  [font=\footnotesize] [align=left] {$x_1$};
     \draw (70.75,86.08) node [anchor=north] [inner sep=0.75pt]  [font=\footnotesize] [align=left] {$x_2$};
     \draw (79.78,20.3) node [anchor=south west] [inner sep=0.75pt]  [font=\footnotesize] [align=left] {$x_3$};
     \draw (127.88,20.3) node [anchor=south east] [inner sep=0.75pt]  [font=\footnotesize] [align=left] {$x_5$};
     \draw (175.99,20.3) node [anchor=south east] [inner sep=0.75pt]  [font=\footnotesize] [align=left] {$x_7$};
     \draw (103.83,116.59) node [anchor=north west][inner sep=0.75pt]  [font=\footnotesize] [align=left] {$x_4$};
     \draw (151.94,116.59) node [anchor=north east] [inner sep=0.75pt]  [font=\footnotesize] [align=left] {$x_6$};
     \draw (188.02,86.08) node [anchor=north] [inner sep=0.75pt]  [font=\footnotesize] [align=left] {$x_8$};
     \draw (219.1,86.08) node [anchor=north west][inner sep=0.75pt]  [font=\footnotesize] [align=left] {$x_9$};
     \draw (332.34,20.3) node [anchor=south west] [inner sep=0.75pt]  [font=\footnotesize] [align=left] {$x_3$};
     \draw (380.44,20.3) node [anchor=south east] [inner sep=0.75pt]  [font=\footnotesize] [align=left] {$x_5$};
     \draw (428.55,20.3) node [anchor=south east] [inner sep=0.75pt]  [font=\footnotesize] [align=left] {$x_7$};
     \draw (289.23,86.08) node [anchor=north east] [inner sep=0.75pt]  [font=\footnotesize] [align=left] {$x_1$};
     \draw (320.31,86.08) node [anchor=north] [inner sep=0.75pt]  [font=\footnotesize] [align=left] {$x_2$};
     \draw (356.39,116.59) node [anchor=north west][inner sep=0.75pt]  [font=\footnotesize] [align=left] {$x_4$};
     \draw (404.5,116.59) node [anchor=north east] [inner sep=0.75pt]  [font=\footnotesize] [align=left] {$x_6$};
     \draw (440.58,86.08) node [anchor=north] [inner sep=0.75pt]  [font=\footnotesize] [align=left] {$x_8$};
     \draw (471.66,86.08) node [anchor=north west][inner sep=0.75pt]  [font=\footnotesize] [align=left] {$x_9$};
     \draw (13.71,83.23) node [anchor=north east] [inner sep=0.75pt]  [font=\footnotesize] [align=left] {$A$};
     \draw (87.83,118.34) node [anchor=east] [inner sep=0.75pt]  [font=\footnotesize] [align=left] {$D$};
     \draw (170.06,83.23) node [anchor=north east] [inner sep=0.75pt]  [font=\footnotesize] [align=left] {$B$};
     \draw (58.35,18.55) node [anchor=east] [inner sep=0.75pt]  [font=\footnotesize] [align=left] {$C$};
     \draw (266.27,83.23) node [anchor=north east] [inner sep=0.75pt]  [font=\footnotesize] [align=left] {$A$};
     \draw (422.62,83.23) node [anchor=north east] [inner sep=0.75pt]  [font=\footnotesize] [align=left] {$B$};
     \draw (310.91,18.55) node [anchor=east] [inner sep=0.75pt]  [font=\footnotesize] [align=left] {$C$};
     \draw (340.39,118.34) node [anchor=east] [inner sep=0.75pt]  [font=\footnotesize] [align=left] {$D$};

     \end{tikzpicture}
\end{center}

We state the stability result of Kopylov's theorem as follows.
\begin{lemma}\label{Lemma:Stablity-Kopylov2}
Let $k\geq 5$ and $\mathcal{F}=\{H_k(1),H_k(M_2),H_k(P_3)\}\cup \mathcal{C}_{2k+1}$.
Let $G$ be a 2-connected graph.
If
\begin{equation}\label{bound for 2-connected}
e(G)>\max\{h(n,2k,k-1),h(n,2k+1,2)\}
\end{equation}
then $G$ contains a copy of  $F\in \mathcal{F}$.
\end{lemma}

Thirdly, we need a refinement of Pos\'{a}'s lemma.
The following is the famous Pos\'{a}'s lemma.
\begin{lemma}[\cite{Posalemma}]\label{posa lemma}
Let $G$ be a 2-connected graph and $xPy$ be a path in $G$ of length $m$. Then
$$c(G)\geq \min\{m+1,d_P(x)+d_P(y)\}.$$
\end{lemma}

We state the refinement of Pos\'{a}'s lemma as follows.

\begin{mydef}\label{fig2}
For integers $k\geq 3$ and $t\geq 2$,
let $P_{2k+t}=x_1\ldots x_{2k+t}$ be a path on $2k +t$ vertices. Let $A=\{x_1,\ldots,x_{k}\}$,
$B=\{x_{2k+t-k+1},\ldots,x_{2k+t}\}$, and $C=\{x_{k+1},\ldots,x_{2k+t-k}\}$. Denote by $F_k(t)$
the graph obtained from $P_{2k+t}$ by adding all edges inside $A$ and $B$, all edges
between $x_k$ and $B$, and all edges between $x_{2k+t-k+1}$ and $A$ (see Figure \ref{fig2}).
\end{mydef}
\begin{center}
\tikzset{every picture/.style={line width=0.75pt}} 
\begin{tikzpicture}[x=0.75pt,y=0.75pt,yscale=-1,xscale=1]

\draw    (53.68,90.29) -- (114.09,90.29) ;
\draw [shift={(114.09,90.29)}, rotate = 0] [color={rgb, 255:red, 0; green, 0; blue, 0 }  ][fill={rgb, 255:red, 0; green, 0; blue, 0 }  ][line width=0.75]      (0, 0) circle [x radius= 1.34, y radius= 1.34]   ;
\draw [shift={(53.68,90.29)}, rotate = 0] [color={rgb, 255:red, 0; green, 0; blue, 0 }  ][fill={rgb, 255:red, 0; green, 0; blue, 0 }  ][line width=0.75]      (0, 0) circle [x radius= 1.34, y radius= 1.34]   ;
\draw    (315.45,90.29) -- (375.85,90.29) ;
\draw [shift={(375.85,90.29)}, rotate = 0] [color={rgb, 255:red, 0; green, 0; blue, 0 }  ][fill={rgb, 255:red, 0; green, 0; blue, 0 }  ][line width=0.75]      (0, 0) circle [x radius= 1.34, y radius= 1.34]   ;
\draw [shift={(315.45,90.29)}, rotate = 0] [color={rgb, 255:red, 0; green, 0; blue, 0 }  ][fill={rgb, 255:red, 0; green, 0; blue, 0 }  ][line width=0.75]      (0, 0) circle [x radius= 1.34, y radius= 1.34]   ;
\draw    (114.09,90.29) -- (124.16,61.54) -- (134.23,32.79) ;
\draw [shift={(134.23,32.79)}, rotate = 289.3] [color={rgb, 255:red, 0; green, 0; blue, 0 }  ][fill={rgb, 255:red, 0; green, 0; blue, 0 }  ][line width=0.75]      (0, 0) circle [x radius= 1.34, y radius= 1.34]   ;
\draw [shift={(114.09,90.29)}, rotate = 289.3] [color={rgb, 255:red, 0; green, 0; blue, 0 }  ][fill={rgb, 255:red, 0; green, 0; blue, 0 }  ][line width=0.75]      (0, 0) circle [x radius= 1.34, y radius= 1.34]   ;
\draw    (174.5,119.04) -- (134.23,32.79) ;
\draw [shift={(134.23,32.79)}, rotate = 244.97] [color={rgb, 255:red, 0; green, 0; blue, 0 }  ][fill={rgb, 255:red, 0; green, 0; blue, 0 }  ][line width=0.75]      (0, 0) circle [x radius= 1.34, y radius= 1.34]   ;
\draw [shift={(174.5,119.04)}, rotate = 244.97] [color={rgb, 255:red, 0; green, 0; blue, 0 }  ][fill={rgb, 255:red, 0; green, 0; blue, 0 }  ][line width=0.75]      (0, 0) circle [x radius= 1.34, y radius= 1.34]   ;
\draw    (255.04,119.04) -- (214.77,119.04) ;
\draw [shift={(214.77,119.04)}, rotate = 180] [color={rgb, 255:red, 0; green, 0; blue, 0 }  ][fill={rgb, 255:red, 0; green, 0; blue, 0 }  ][line width=0.75]      (0, 0) circle [x radius= 1.34, y radius= 1.34]   ;
\draw [shift={(255.04,119.04)}, rotate = 180] [color={rgb, 255:red, 0; green, 0; blue, 0 }  ][fill={rgb, 255:red, 0; green, 0; blue, 0 }  ][line width=0.75]      (0, 0) circle [x radius= 1.34, y radius= 1.34]   ;
\draw    (315.45,90.29) -- (295.31,32.79) ;
\draw [shift={(295.31,32.79)}, rotate = 250.7] [color={rgb, 255:red, 0; green, 0; blue, 0 }  ][fill={rgb, 255:red, 0; green, 0; blue, 0 }  ][line width=0.75]      (0, 0) circle [x radius= 1.34, y radius= 1.34]   ;
\draw [shift={(315.45,90.29)}, rotate = 250.7] [color={rgb, 255:red, 0; green, 0; blue, 0 }  ][fill={rgb, 255:red, 0; green, 0; blue, 0 }  ][line width=0.75]      (0, 0) circle [x radius= 1.34, y radius= 1.34]   ;
\draw    (174.5,119.04) -- (214.77,119.04) ;
\draw [shift={(214.77,119.04)}, rotate = 0] [color={rgb, 255:red, 0; green, 0; blue, 0 }  ][fill={rgb, 255:red, 0; green, 0; blue, 0 }  ][line width=0.75]      (0, 0) circle [x radius= 1.34, y radius= 1.34]   ;
\draw [shift={(174.5,119.04)}, rotate = 0] [color={rgb, 255:red, 0; green, 0; blue, 0 }  ][fill={rgb, 255:red, 0; green, 0; blue, 0 }  ][line width=0.75]      (0, 0) circle [x radius= 1.34, y radius= 1.34]   ;
\draw    (255.04,119.04) -- (295.31,32.79) ;
\draw [shift={(295.31,32.79)}, rotate = 295.03] [color={rgb, 255:red, 0; green, 0; blue, 0 }  ][fill={rgb, 255:red, 0; green, 0; blue, 0 }  ][line width=0.75]      (0, 0) circle [x radius= 1.34, y radius= 1.34]   ;
\draw [shift={(255.04,119.04)}, rotate = 295.03] [color={rgb, 255:red, 0; green, 0; blue, 0 }  ][fill={rgb, 255:red, 0; green, 0; blue, 0 }  ][line width=0.75]      (0, 0) circle [x radius= 1.34, y radius= 1.34]   ;
\draw    (134.23,32.79) -- (375.85,90.29) ;
\draw    (53.68,90.29) -- (93.96,61.54) -- (134.23,32.79) ;
\draw    (53.68,90.29) -- (295.31,32.79) ;
\draw    (114.09,90.29) -- (295.31,32.79) ;
\draw    (375.85,90.29) -- (335.58,61.54) -- (295.31,32.79) ;
\draw    (315.45,90.29) -- (134.23,32.79) ;
\draw    (134.23,32.79) -- (295.31,32.79) ;
\draw   (156.69,119.04) .. controls (156.69,113.93) and (182.69,109.78) .. (214.77,109.78) .. controls (246.85,109.78) and (272.85,113.93) .. (272.85,119.04) .. controls (272.85,124.16) and (246.85,128.31) .. (214.77,128.31) .. controls (182.69,128.31) and (156.69,124.16) .. (156.69,119.04) -- cycle ;

\draw (214.77,131.83) node [anchor=north] [inner sep=0.75pt]   [align=left] {Figure \ref{fig2}. $F_3(3)$};
\draw (133,35.03) node [anchor=south east] [inner sep=0.75pt]  [font=\footnotesize] [align=left] {$x_{k}$};
\draw (297.6,35.03) node [anchor=south west] [inner sep=0.75pt]  [font=\footnotesize] [align=left] {$x_{k+t+1}$};
\draw (92.73,63.89) node [anchor=south east] [inner sep=0.75pt]  [font=\footnotesize] [align=left] {$A$};
\draw (336.81,63.89) node [anchor=south west] [inner sep=0.75pt]  [font=\footnotesize] [align=left] {$B$};
\draw (214.77,110.13) node [anchor=south] [inner sep=0.75pt]  [font=\footnotesize] [align=left] {$C$};
\end{tikzpicture}
\end{center}
\begin{lemma}\label{extend posa lemma}
 Let $k\geq 4$. Let $G$ be a 2-connected graph with $c(G)\leq 2k$ and $P=x_1\ldots x_m$ be a maximum path with $m\geq2k+1$. If
\begin{equation}\label{eq1 for lemma}
d_P(x_1)= k \mbox{ and }d_P(x_{m})= k
\end{equation}
and for each path $P^\prime$ on $m$ vertices in  $G[V(P)]$ with end-vertices $u,v\in N_P(x_1)\cup N_P(x_m)$, we have
\begin{equation}\label{eq2 for lemma}
 d_{P^\prime}(u)= k \mbox{ and } d_{P^\prime}(v)= k,
\end{equation}
then $H_k(s)\subseteq G[V(P)] $ for some $s$ when $m=2k+1$, or $F_k(t) \subseteq G[V(P)]$ for some $t$ when $m\geq 2k+2$.
\end{lemma}
\begin{proof}
Let $N_P^-(x_1)=\{x_i:x_{i+1}\in N_P(x_1)\}$ and $N_P^+(x_{m})=\{x_i:x_{i-1}\in N_P(x_{m})\}$.
    Clearly, we have $|N_P^-(x_1)|=d_P(x_1)$ and $|N_P^+(x_{m})|=d_P(x_m)$. Since $c(G)\leq 2k$, we have
\begin{equation}\label{neighbour of end vertices}
N_P^-(x_1)\cap N_P(x_{m})=\emptyset \mbox{ and }N_P^+(x_{m})\cap N_P(x_1)=\emptyset.
\end{equation}

We will prove the lemma in the following three cases.

\medskip

\noindent{\bf Case 1.} There exist $1\leq i<j\leq m$ such that $x_i\in N_P(x_{m})$ and $x_j\in N_P(x_1)$.

\medskip

We say a pair $(i,j)$ a {\it crossing pair} if $x_i\in N_P(x_{m})$, $x_j\in N_P(x_1)$
and $x_{\ell}\notin  N_P(x_1)\cup N_P(x_{m})$ for $i<\ell<j$.
Let $(i,j)$ be a crossing pair and $U_{i,j}=\{x_{\ell}:i<\ell<j\}$.
We have the following observations:
\begin{observation}\label{ob1}
     $V(x_1Px_i)\cup V(x_jPx_m)=\{x_1\}\cup N_P(x_1)\cup(N_P^+(x_{m})\setminus \{x_{i+1}\})$.

\end{observation}
\begin{proof}
     By (\ref{neighbour of end vertices}), we have $(\{x_1\}\cup N_P(x_1))\cap(N_P^+(x_{m})\setminus \{x_{i+1}\})=\emptyset$, and hence $|\{x_1\}\cup N_P(x_1)\cup(N_P^+(x_{m})\setminus \{x_{i+1}\})| = 1+d_P(x_1)+d_P(x_{m})-1= 2k$.
     Since $x_1x_jPx_{m}x_iPx_1$ is a cycle on at most $2k$ vertices, the result follows.
\end{proof}
\begin{observation}\label{ob2}
We have $j-i=m-2k+1$. If there exists a crossing pair $(i^\prime,j^\prime)\neq (i,j)$,
then $j^\prime=i^\prime+2$, $j=i+2$, and $m=2k+1$.
\end{observation}
\begin{proof}
The sum of lengths of $x_1Px_i$ and $x_jPx_{m}$ is $2k-2$. The length of $P$ is $m-1$. Hence $j-i=(m-1)-(2k-2)=m-2k+1$.

Suppose $j^\prime > i^\prime+2$ such that $x_{\ell}\notin  N_P(x_1)\cup N_P(x_{m})$ for $i^\prime<\ell<j^\prime$ with $(i^\prime,j^\prime)\neq (i,j)$. By Observation \ref{ob1}, we have $x_{i^\prime+2}\in N_P^+(x_{m})\setminus \{x_{i+1}\}$, implying $x_{i^\prime+1}\in N_P(x_{m})$, which is a contradiction.
Thus, $j^\prime = i^\prime+2$, and by symmetry, we have $j=i+2$, implying $m=2k+1$.
\end{proof}

Let $s=\min\{i:x_{i+1}\in N_P(x_m)\}$ and $s^\prime=\min\{i: x_{m-i}\in N_P(x_1)\}$.
For $2\leq q\leq s$, since $P^\prime=x_qPx_1x_{q+1}Px_{m}$ is a path on $m+1$ vertices, by (\ref{eq2 for lemma}), we have $d_{P^\prime}(x_q)= k$.
Suppose that there exists a vertex $x_\ell\in U_{i,j}\cap N_P(x_q)$.
Then $x_q P^\prime x_i x_m P^\prime x_\ell x_q$ is a cycle on at least $2k+1$ vertices,
which is a contradiction. Hence, $U_{i,j}\cap N_P(x_q)=\emptyset$.
Similar to (\ref{neighbour of end vertices}), we have $N_P^+(x_{m})\cap N_P(x_q)=\emptyset$
(clearly, $N_P^+(x_{m})=N_{P^\prime}^+(x_{m})$ $N_P(x_q)=N_{P^\prime}(x_q)$).
We can further get $\{x_q\}\cup N_P(x_q)\cup(N_P^+(x_{m})\setminus \{x_{i+1}\})= V(P)\setminus U_{i,j} = \{x_1\}\cup N_P(x_1)\cup(N_P^+(x_{m})\setminus \{x_{i+1}\})$
as in Observation 1. For $2 \leq q \leq s$,  we have
     \begin{equation}\label{vertices in A and B}
          \{x_q\}\cup N_P(x_q)=\{x_1\}\cup N_P(x_1).
          \end{equation}
          Similarly, for $m-s^\prime+1 \leq q \leq m-1$,  we have
          \begin{equation}\label{vertices in A+ and B+}
               \{x_q\}\cup N_P(x_q)=\{x_m\}\cup N_P(x_m).
               \end{equation}
\begin{claim}\label{claim2}
$x_1$ is not adjacent to both $x_{t}$ and $x_{t+1}$ for $t\geq s+1$ and $x_m$
is not joint to both of $x_{t}$ and $x_{t-1}$ for $t\leq m-s^\prime$.
\end{claim}
\begin{proof}
Suppose that $x_1$ is adjacent to both $x_t$ and $x_{t+1}$. Then by (\ref{vertices in A and B}),
$x_s$ is adjacent to $x_t$ and $x_{t+1}$, and so $x_{m}Px_{t+1}x_1Px_{s}x_tPx_{s+1}x_{m}$ is a
cycle on at least $2k+1$ vertices, a contradiction. Similarly,  $x_m$ is not joint to both
$x_{t}$ and $x_{t-1}$ for $t\leq m-s^\prime$. The proof is complete.
\end{proof}
Now we divide the proof into following cases according to the values of $m$.

\medskip

\noindent{\bf Subcase 1.1.} $m=2k+1$.

\medskip

Let $A=\{x_1,\ldots,x_{s}\}$ and $B=\{x_{2k+1-s^\prime+1},\ldots,x_{2k+1}\}$. Let
$C=\{x_{s+1}, x_{s+3}\ldots,x_p\}$ and $D=\{x_{s+2}, x_{s+4}\ldots,x_q\}$,
where $p=s+2\left\lceil (2k+1-s-s^\prime)/2\right\rceil-1$ and $q=s+2\left\lfloor (2k+1-s-s^\prime)/2\right\rfloor$.
Since $x_1$ is not connected to the vertices in $B$ by the definition of $s^\prime$,
and is not joint to consecutive vertices in $V(P)\backslash (A\cup B)$ by
Claim \ref{claim2}, we can see in $P$, $x_1$ is joint to at most $s-1$ vertices
in $A$ and at most $(2k+1-s-s^\prime+1)/2$ vertices in $P\backslash (A\cup B)$.
Thus, $s-1+(2k+1-s-s^\prime+1)/2\geq d_P(x_1)=k$.
Similarly, for $x_{m}$ we have $s^\prime-1+(2k+1-s-s^\prime+1)/2\geq d_P(x_m)=k$.
Therefore, $s=s^\prime$ and $x_1$ are jointed to each vertex of $A\cup C$,
and $x_{m}$ is jointed to each vertex of $B\cup C$.
By (\ref{vertices in A and B}) and (\ref{vertices in A+ and B+}), $G[A]$ and $G[B]$
are complete graphs, and $G[A\cup B, C]$ is a complete bipartite graph. Thus, $H_k(s)\subseteq G[V(P)]$.

\medskip

\noindent{\bf Subcase 1.2.} $m=2k+t\geq 2k+2$.

\medskip

Let $A=\{x_1,\ldots,x_{k}\}$, $B=\{x_{2k+t-k+1},\ldots,x_{2k+t}\}$, and $C=\{x_{k},\ldots,x_{2k+t-k+1}\}$.
By Observation \ref{ob2},
there is only one crossing pair $(i,j)$.
If there exists $t\in(j,m)$ such that $x_1$ is joint with $x_t$, it follows from Claim \ref{claim2} that $t\geq j+2$. By Observation \ref{ob1}, $x_{j+1}\in N_P^+(x_{m})\setminus \{x_{i+1}\}$, hence $x_j\in N_P(x_{m})$, which implies that there exists another crossing pair $(j,t)$ or a cycle of length at least $2k+1$, both are contradictions.
Therefore, it follows from Observation 1 that $x_{q}\in N_P(x_1)$  for $2\leq q\leq i$ and $x_q\in N_P(x_m)$ for $j\leq q\leq m-1$. Thus, $i=k$ and $j=k+t+1$.
Furthermore, by (\ref{vertices in A and B}) and (\ref{vertices in A+ and B+}), it is easy to see that $F_k(t) \subseteq G[V(P)]$ ($G[A]$ and $G[B]$ are complete graphs).
Thus, we finish the proof of Case 1.

\medskip

The following two cases are well-known results of Pos\'{a}. We include the proofs for completeness.

\medskip

\noindent{\bf Case 2.} Case 1 does not occur and there exists $i$ such that $x_i\in N_P(x_{1})$ and $x_i\in N_P(x_{m})$.

\medskip

Since $G$ is 2-connected, there exists a path $Q$ with $V(Q)\cap V(P)=\{x_s,x_t\}$ and $1\leq s<i<t\leq m$. Let
$$
p=\min\{i>s:x_i\in N_P(x_1)\} \mbox{ and } q=\max\{i<t:x_i\in N_P(x_{m})\}.
$$
Then $x_1Px_sQx_tPx_{m}x_qPx_px_1$ is the cycle that contains $\{x_1\}\cup N_P(x_1)\cup x_{m} \cup N_P(x_{m})$. Hence $G$ contains a cycle on at least $2k+1$ vertices, a contradiction.

\medskip

\noindent{\bf Case 3.} Both Case 1 and Case 2 do not occur.

\medskip

Let $g=\max\{i:x_i\in N_P(x_1)\}$ and $h=\min\{i:i\in N_P(x_{m})\}$. Since $G$ is 2-connected, there exists a path $Q_1$ such that it intersects $P$ with exactly two vertices $x_{s_1}$, $x_{t_1}$, and $s_1<g<t_1$. Choose such a path with $t_1$ as large as possible. If $t_1>h$, then we stop. If $t_1<h$, then we choose a path $Q_2$ such that it intersects $P$ with exactly two vertices $x_{s_2}$, $x_{t_2}$ and $s_2<t_1<t_2$. Choose a path with $t_2$ as large as possible. Since we choose $t_1$ as large as possible, $Q_1\cup Q_2 =\emptyset$. If $t_2>h$, then we stop. Otherwise, we may continue this procedure and obtain a path $Q_{\ell}$ such that it intersects $P$ with exactly two vertices $x_{s_{\ell}}$, $x_{t_{\ell}}$ and $s_{\ell}<t_{\ell-1}<h<t_{\ell}$. Moreover, for any $Q_i$ and $Q_j$ with $i<j$, either $Q_i\cap Q_j =\emptyset$ or $Q_i\cap Q_j=s_j=t_{i}$ for $j=i+2$. Let
$$
i_0=\min\{i>s_1:x_i\in N_P(x_1)\}  \mbox{ and } j_0=\max\{i<t_\ell:x_i\in N_P(x_{m})\}.
$$
If $\ell$ is odd, since Case 2 does not occur and $d_P(x_1)+d_P(x_{m})\geq 2k$,
$$
x_1Px_{s_1}Q_1x_{t_1}Px_{s_3}Q_3x_{t_3}P x_{s_{\ell}}Q_{\ell}x_{t_{\ell}}Px_{m}x_{j_0}Px_{t_{\ell-1}}Q_{\ell-1}x_{s_{\ell-1}}Px_{t_2}Q_2x_{s_2}P x_{i_0}x_1
$$
is a cycle on at least $2k+2$ vertices.
If $\ell$ is even, then
$$
x_1Px_{s_1}Q_1x_{t_1}P x_{s_3}Q_3x_{t_3} P x_{s_{\ell-1}}Q_{\ell-1}x_{t_{\ell-1}}P x_{j_0}x_{m}P x_{t_{\ell}}Q_{\ell}x_{s_{\ell}}P x_{t_4}Q_4x_{s_4}P x_{t_2}Q_2x_{s_2}P x_{i_0} x_1
$$
is a cycle on at least $2k+2$ vertices. Both are contradictions.
The proof is complete.
\end{proof}

\section{Several definitions and two lemmas}\label{Sec:3}
The graphs in Definitions~\ref{fg}, \ref{fig:1}, and \ref{fig2} have good properties for
the proof of our main theorem. We will present those good properties in the following lemmas.

For $H_k(P_3)$, let the smaller part of the vertex sets of $K_{k-1,k+2}$ be
$B=\{b_1,b_2,\ldots,b_{k-1}\}$, the larger one be partitioned into two sets
$A$ and $C$ such that $A=\{a_1,a_2,a_3\}$, $C=\{c_1,c_2,\ldots,c_{k-1}\}$ and
$a_1a_2a_3$ is the $P_3$ embedded in $K_{k-1,k+2}$ (see Figure \ref{fg}).

For $H_k(M_2)$, let the smaller part of the vertex sets of $K_{k-1,k+2}$ be
$C=\{x_3,x_5,\ldots,x_{2k-1}\}$, the larger one be partitioned into three
sets $A$, $B$, and $D$ such that $A=\{x_1,x_2\}$, $B=\{x_{2k},x_{2k+1}\}$,
$D=\{x_{4}$, $x_{6},\ldots,x_{2k-2}\}$, and $\{x_1x_2,x_{2k}x_{2k+1}\}$ is
the $M_2$ embedded in $K_{k-1,k+2}$  (see Figure \ref{fg}).

     \tikzset{every picture/.style={line width=0.75pt}} 
\begin{center}
     \begin{tikzpicture}[x=0.75pt,y=0.75pt,yscale=-1,xscale=1]

     \draw    (66.06,151.08) -- (127.92,151.08) ;
     \draw [shift={(127.92,151.08)}, rotate = 0] [color={rgb, 255:red, 0; green, 0; blue, 0 }  ][fill={rgb, 255:red, 0; green, 0; blue, 0 }  ][line width=0.75]      (0, 0) circle [x radius= 1.34, y radius= 1.34]   ;
     \draw [shift={(96.99,151.08)}, rotate = 0] [color={rgb, 255:red, 0; green, 0; blue, 0 }  ][fill={rgb, 255:red, 0; green, 0; blue, 0 }  ][line width=0.75]      (0, 0) circle [x radius= 1.34, y radius= 1.34]   ;
     \draw [shift={(66.06,151.08)}, rotate = 0] [color={rgb, 255:red, 0; green, 0; blue, 0 }  ][fill={rgb, 255:red, 0; green, 0; blue, 0 }  ][line width=0.75]      (0, 0) circle [x radius= 1.34, y radius= 1.34]   ;
     \draw [draw opacity=0]   (158.85,151.08) -- (220.7,151.08) ;
     \draw [shift={(220.7,151.08)}, rotate = 0] [draw opacity=0][line width=0.75]      (0, 0) circle [x radius= 1.34, y radius= 1.34]   ;
     \draw [shift={(189.77,151.08)}, rotate = 0] [draw opacity=0][line width=0.75]      (0, 0) circle [x radius= 1.34, y radius= 1.34]   ;
     \draw [shift={(158.85,151.08)}, rotate = 0] [draw opacity=0][line width=0.75]      (0, 0) circle [x radius= 1.34, y radius= 1.34]   ;
     \draw [draw opacity=0]   (189.77,151.08) -- (251.63,151.08) ;
     \draw [shift={(251.63,151.08)}, rotate = 0] [draw opacity=0][line width=0.75]      (0, 0) circle [x radius= 1.34, y radius= 1.34]   ;
     \draw [shift={(220.7,151.08)}, rotate = 0] [draw opacity=0][line width=0.75]      (0, 0) circle [x radius= 1.34, y radius= 1.34]   ;
     \draw [shift={(189.77,151.08)}, rotate = 0] [draw opacity=0][line width=0.75]      (0, 0) circle [x radius= 1.34, y radius= 1.34]   ;
     \draw [draw opacity=0]   (103.18,61.36) -- (165.03,61.36) ;
     \draw [shift={(165.03,61.36)}, rotate = 0] [draw opacity=0][line width=0.75]      (0, 0) circle [x radius= 1.34, y radius= 1.34]   ;
     \draw [shift={(134.1,61.36)}, rotate = 0] [draw opacity=0][line width=0.75]      (0, 0) circle [x radius= 1.34, y radius= 1.34]   ;
     \draw [shift={(103.18,61.36)}, rotate = 0] [draw opacity=0][line width=0.75]      (0, 0) circle [x radius= 1.34, y radius= 1.34]   ;
     \draw    (103.18,61.36) -- (189.77,151.08) -- (134.1,61.36) -- (66.06,151.08) -- (195.96,61.36) -- (220.7,151.08) -- (103.18,61.36) -- (158.85,151.08) ;
     \draw    (165.03,61.36) -- (251.63,151.08) ;
     \draw [draw opacity=0]   (134.1,61.36) -- (195.96,61.36) ;
     \draw [shift={(195.96,61.36)}, rotate = 0] [draw opacity=0][line width=0.75]      (0, 0) circle [x radius= 1.34, y radius= 1.34]   ;
     \draw [shift={(165.03,61.36)}, rotate = 0] [draw opacity=0][line width=0.75]      (0, 0) circle [x radius= 1.34, y radius= 1.34]   ;
     \draw [shift={(134.1,61.36)}, rotate = 0] [draw opacity=0][line width=0.75]      (0, 0) circle [x radius= 1.34, y radius= 1.34]   ;
     \draw    (66.06,151.08) -- (103.18,61.36) -- (96.99,151.08) -- (134.1,61.36) -- (127.92,151.08) -- (165.03,61.36) -- (158.85,151.08) -- (195.96,61.36) -- (189.77,151.08) -- (165.03,61.36) -- (220.7,151.08) -- (134.1,61.36) -- (251.63,151.08) -- (195.96,61.36) -- (127.92,151.08) -- (103.18,61.36) -- (251.63,151.08) ;
     \draw    (158.85,151.08) -- (134.1,61.36) ;
     \draw    (66.06,151.08) -- (165.03,61.36) -- (96.99,151.08) -- (195.96,61.36) ;
     \draw   (81.32,60.17) .. controls (81.32,52.74) and (111.78,46.71) .. (149.36,46.71) .. controls (186.94,46.71) and (217.4,52.74) .. (217.4,60.17) .. controls (217.4,67.6) and (186.94,73.63) .. (149.36,73.63) .. controls (111.78,73.63) and (81.32,67.6) .. (81.32,60.17) -- cycle ;
     \draw   (140.29,151.08) .. controls (140.29,145.68) and (167.98,141.3) .. (202.14,141.3) .. controls (236.31,141.3) and (264,145.68) .. (264,151.08) .. controls (264,156.47) and (236.31,160.85) .. (202.14,160.85) .. controls (167.98,160.85) and (140.29,156.47) .. (140.29,151.08) -- cycle ;
     \draw [draw opacity=0]   (326.06,151.08) -- (356.99,151.08) ;
     \draw [shift={(356.99,151.08)}, rotate = 0] [draw opacity=0][line width=0.75]      (0, 0) circle [x radius= 1.34, y radius= 1.34]   ;
     \draw [shift={(341.53,151.08)}, rotate = 0] [draw opacity=0][line width=0.75]      (0, 0) circle [x radius= 1.34, y radius= 1.34]   ;
     \draw [shift={(326.06,151.08)}, rotate = 0] [draw opacity=0][line width=0.75]      (0, 0) circle [x radius= 1.34, y radius= 1.34]   ;
     \draw [draw opacity=0]   (387.92,151.08) -- (449.77,151.08) ;
     \draw [shift={(449.77,151.08)}, rotate = 0] [draw opacity=0][line width=0.75]      (0, 0) circle [x radius= 1.34, y radius= 1.34]   ;
     \draw [shift={(418.85,151.08)}, rotate = 0] [draw opacity=0][line width=0.75]      (0, 0) circle [x radius= 1.34, y radius= 1.34]   ;
     \draw [shift={(387.92,151.08)}, rotate = 0] [draw opacity=0][line width=0.75]      (0, 0) circle [x radius= 1.34, y radius= 1.34]   ;
     \draw [draw opacity=0]   (480.7,151.08) -- (511.63,151.08) ;
     \draw [shift={(511.63,151.08)}, rotate = 0] [draw opacity=0][line width=0.75]      (0, 0) circle [x radius= 1.34, y radius= 1.34]   ;
     \draw [shift={(496.16,151.08)}, rotate = 0] [draw opacity=0][line width=0.75]      (0, 0) circle [x radius= 1.34, y radius= 1.34]   ;
     \draw [shift={(480.7,151.08)}, rotate = 0] [draw opacity=0][line width=0.75]      (0, 0) circle [x radius= 1.34, y radius= 1.34]   ;
     \draw [draw opacity=0]   (363.18,61.36) -- (425.03,61.36) ;
     \draw [shift={(425.03,61.36)}, rotate = 0] [draw opacity=0][line width=0.75]      (0, 0) circle [x radius= 1.34, y radius= 1.34]   ;
     \draw [shift={(394.1,61.36)}, rotate = 0] [draw opacity=0][line width=0.75]      (0, 0) circle [x radius= 1.34, y radius= 1.34]   ;
     \draw [shift={(363.18,61.36)}, rotate = 0] [draw opacity=0][line width=0.75]      (0, 0) circle [x radius= 1.34, y radius= 1.34]   ;
     \draw    (363.18,61.36) -- (449.77,151.08) -- (394.1,61.36) -- (326.06,151.08) -- (455.96,61.36) -- (480.7,151.08) -- (363.18,61.36) -- (418.85,151.08) ;
     \draw    (425.03,61.36) -- (511.63,151.08) ;
     \draw [draw opacity=0]   (394.1,61.36) -- (455.96,61.36) ;
     \draw [shift={(455.96,61.36)}, rotate = 0] [draw opacity=0][line width=0.75]      (0, 0) circle [x radius= 1.34, y radius= 1.34]   ;
     \draw [shift={(425.03,61.36)}, rotate = 0] [draw opacity=0][line width=0.75]      (0, 0) circle [x radius= 1.34, y radius= 1.34]   ;
     \draw [shift={(394.1,61.36)}, rotate = 0] [draw opacity=0][line width=0.75]      (0, 0) circle [x radius= 1.34, y radius= 1.34]   ;
     \draw    (326.06,151.08) -- (363.18,61.36) -- (356.99,151.08) -- (394.1,61.36) -- (387.92,151.08) -- (425.03,61.36) -- (418.85,151.08) -- (455.96,61.36) -- (449.77,151.08) -- (425.03,61.36) -- (480.7,151.08) -- (394.1,61.36) -- (511.63,151.08) -- (455.96,61.36) -- (387.92,151.08) -- (363.18,61.36) -- (511.63,151.08) ;
     \draw    (418.85,151.08) -- (394.1,61.36) ;
     \draw    (326.06,151.08) -- (425.03,61.36) -- (356.99,151.08) -- (455.96,61.36) ;
     \draw   (341.32,60.17) .. controls (341.32,52.74) and (371.78,46.71) .. (409.36,46.71) .. controls (446.94,46.71) and (477.4,52.74) .. (477.4,60.17) .. controls (477.4,67.6) and (446.94,73.63) .. (409.36,73.63) .. controls (371.78,73.63) and (341.32,67.6) .. (341.32,60.17) -- cycle ;
     \draw   (370,150.65) .. controls (370,145.49) and (392.39,141.3) .. (420,141.3) .. controls (447.61,141.3) and (470,145.49) .. (470,150.65) .. controls (470,155.81) and (447.61,160) .. (420,160) .. controls (392.39,160) and (370,155.81) .. (370,150.65) -- cycle ;
     \draw    (326.06,151.08) -- (356.99,151.08) ;
     \draw    (480.7,151.08) -- (511.63,151.08) ;

     \draw (103.18,61.36) node [anchor=south west] [inner sep=0.75pt]  [font=\small] [align=left] {$b_1$};
     \draw (134.1,61.36) node [anchor=south] [inner sep=0.75pt]  [font=\small] [align=left] {$b_2$};
     \draw (172.41,61.36) node [anchor=south east] [inner sep=0.75pt]  [font=\small] [align=left] {$b_3$};
     \draw (195.96,61.36) node [anchor=south east] [inner sep=0.75pt]  [font=\small] [align=left] {$b_4$};
     \draw (62.32,148.71) node [anchor=south east] [inner sep=0.75pt]  [font=\normalsize] [align=left] {$A$};
     \draw (78.43,62) node [anchor=south east] [inner sep=0.75pt]  [font=\normalsize] [align=left] {$B$};
     \draw (264,151.08) node [anchor=south west] [inner sep=0.75pt]  [font=\normalsize] [align=left] {$C$};
     \draw (65.82,155.71) node [anchor=north] [inner sep=0.75pt]  [font=\small] [align=left] {$a_1$};
     \draw (97.82,155.71) node [anchor=north] [inner sep=0.75pt]  [font=\small] [align=left] {$a_2$};
     \draw (158.85,160.61) node [anchor=north] [inner sep=0.75pt]  [font=\small] [align=left] {$c_1$};
     \draw (189.77,160.61) node [anchor=north] [inner sep=0.75pt]  [font=\small] [align=left] {$c_2$};
     \draw (220.7,160.61) node [anchor=north] [inner sep=0.75pt]  [font=\small] [align=left] {$c_3$};
     \draw (251.63,160.61) node [anchor=north] [inner sep=0.75pt]  [font=\small] [align=left] {$c_4$};
     \draw (127.82,155.71) node [anchor=north] [inner sep=0.75pt]  [font=\small] [align=left] {$a_3$};
     \draw (322.32,148.71) node [anchor=south east] [inner sep=0.75pt]  [font=\normalsize] [align=left] {$A$};
     \draw (338.43,62) node [anchor=south east] [inner sep=0.75pt]  [font=\normalsize] [align=left] {$C$};
     \draw (513,148.71) node [anchor=south west] [inner sep=0.75pt]  [font=\normalsize] [align=left] {$B$};
     \draw (326.06,156.61) node [anchor=north] [inner sep=0.75pt]  [font=\small] [align=left] {$x_1$};
     \draw (356.82,156.61) node [anchor=north] [inner sep=0.75pt]  [font=\small] [align=left] {$x_2$};
     \draw (480.7,156.61) node [anchor=north] [inner sep=0.75pt]  [font=\small] [align=left] {$x_{10}$};
     \draw (511.63,156.61) node [anchor=north] [inner sep=0.75pt]  [font=\small] [align=left] {$x_{11}$};
     \draw (387.92,160.61) node [anchor=north] [inner sep=0.75pt]  [font=\small] [align=left] {$x_4$};
     \draw (434.31,145.61) node [anchor=north] [inner sep=0.75pt]  [font=\normalsize] [align=left] {$D$};

     \draw (418.85,160.61) node [anchor=north] [inner sep=0.75pt]  [font=\small] [align=left] {$x_6$};
     \draw (449.77,160.61) node [anchor=north] [inner sep=0.75pt]  [font=\small] [align=left] {$x_8$};
     \draw (363.18,61.36) node [anchor=south west] [inner sep=0.75pt]  [font=\small] [align=left] {$x_3$};
     \draw (394.1,61.36) node [anchor=south] [inner sep=0.75pt]  [font=\small] [align=left] {$x_5$};
     \draw (432.41,61.36) node [anchor=south east] [inner sep=0.75pt]  [font=\small] [align=left] {$x_7$};
     \draw (455.96,61.36) node [anchor=south east] [inner sep=0.75pt]  [font=\small] [align=left] {$x_9$};
     \draw (247,193) node [anchor=north west][inner sep=0.75pt]   [align=left] {{Figure \ref{fg}.}};
     \draw (140.29,177.08) node [anchor=north] [inner sep=0.75pt]   [align=left] {$H_5(P_3)$};
     \draw (419.29,177.08) node [anchor=north] [inner sep=0.75pt]   [align=left] {$H_5(M_2)$};

     \draw [fill={rgb, 255:red, 37; green, 35; blue, 35 }  ,fill opacity=1 ]  (189.77, 151.08) circle [x radius= 1.34, y radius= 1.34]   ;
     \draw [fill={rgb, 255:red, 37; green, 35; blue, 35 }  ,fill opacity=1 ]  (189.77, 151.08) circle [x radius= 1.34, y radius= 1.34]   ;
     \draw [fill={rgb, 255:red, 37; green, 35; blue, 35 }  ,fill opacity=1 ]  (220.7, 151.08) circle [x radius= 1.34, y radius= 1.34]   ;
     \draw [fill={rgb, 255:red, 37; green, 35; blue, 35 }  ,fill opacity=1 ]  (220.7, 151.08) circle [x radius= 1.34, y radius= 1.34]   ;
     \draw [fill={rgb, 255:red, 37; green, 35; blue, 35 }  ,fill opacity=1 ]  (251.63, 151.08) circle [x radius= 1.34, y radius= 1.34]   ;
     \draw [fill={rgb, 255:red, 37; green, 35; blue, 35 }  ,fill opacity=1 ]  (189.77, 151.08) circle [x radius= 1.34, y radius= 1.34]   ;
     \draw [fill={rgb, 255:red, 37; green, 35; blue, 35 }  ,fill opacity=1 ]  (189.77, 151.08) circle [x radius= 1.34, y radius= 1.34]   ;
     \draw [fill={rgb, 255:red, 37; green, 35; blue, 35 }  ,fill opacity=1 ]  (220.7, 151.08) circle [x radius= 1.34, y radius= 1.34]   ;
     \draw [fill={rgb, 255:red, 37; green, 35; blue, 35 }  ,fill opacity=1 ]  (220.7, 151.08) circle [x radius= 1.34, y radius= 1.34]   ;
     \draw [fill={rgb, 255:red, 37; green, 35; blue, 35 }  ,fill opacity=1 ]  (251.63, 151.08) circle [x radius= 1.34, y radius= 1.34]   ;
     \draw [fill={rgb, 255:red, 37; green, 35; blue, 35 }  ,fill opacity=1 ]  (251.63, 151.08) circle [x radius= 1.34, y radius= 1.34]   ;
     \draw [fill={rgb, 255:red, 37; green, 35; blue, 35 }  ,fill opacity=1 ]  (251.63, 151.08) circle [x radius= 1.34, y radius= 1.34]   ;
     \draw [fill={rgb, 255:red, 37; green, 35; blue, 35 }  ,fill opacity=1 ]  (189.77, 151.08) circle [x radius= 1.34, y radius= 1.34]   ;
     \draw [fill={rgb, 255:red, 37; green, 35; blue, 35 }  ,fill opacity=1 ]  (189.77, 151.08) circle [x radius= 1.34, y radius= 1.34]   ;
     \draw [fill={rgb, 255:red, 37; green, 35; blue, 35 }  ,fill opacity=1 ]  (220.7, 151.08) circle [x radius= 1.34, y radius= 1.34]   ;
     \draw [fill={rgb, 255:red, 37; green, 35; blue, 35 }  ,fill opacity=1 ]  (220.7, 151.08) circle [x radius= 1.34, y radius= 1.34]   ;
     \draw [fill={rgb, 255:red, 37; green, 35; blue, 35 }  ,fill opacity=1 ]  (158.85, 151.08) circle [x radius= 1.34, y radius= 1.34]   ;
     \draw [fill={rgb, 255:red, 37; green, 35; blue, 35 }  ,fill opacity=1 ]  (158.85, 151.08) circle [x radius= 1.34, y radius= 1.34]   ;
     \draw [fill={rgb, 255:red, 37; green, 35; blue, 35 }  ,fill opacity=1 ]  (158.85, 151.08) circle [x radius= 1.34, y radius= 1.34]   ;
     \draw [fill={rgb, 255:red, 37; green, 35; blue, 35 }  ,fill opacity=1 ]  (189.77, 151.08) circle [x radius= 1.34, y radius= 1.34]   ;
     \draw [fill={rgb, 255:red, 37; green, 35; blue, 35 }  ,fill opacity=1 ]  (189.77, 151.08) circle [x radius= 1.34, y radius= 1.34]   ;
     \draw [fill={rgb, 255:red, 37; green, 35; blue, 35 }  ,fill opacity=1 ]  (220.7, 151.08) circle [x radius= 1.34, y radius= 1.34]   ;
     \draw [fill={rgb, 255:red, 37; green, 35; blue, 35 }  ,fill opacity=1 ]  (220.7, 151.08) circle [x radius= 1.34, y radius= 1.34]   ;
     \draw [fill={rgb, 255:red, 37; green, 35; blue, 35 }  ,fill opacity=1 ]  (158.85, 151.08) circle [x radius= 1.34, y radius= 1.34]   ;
     \draw [fill={rgb, 255:red, 37; green, 35; blue, 35 }  ,fill opacity=1 ]  (134.1, 61.36) circle [x radius= 1.34, y radius= 1.34]   ;
     \draw [fill={rgb, 255:red, 37; green, 35; blue, 35 }  ,fill opacity=1 ]  (134.1, 61.36) circle [x radius= 1.34, y radius= 1.34]   ;
     \draw [fill={rgb, 255:red, 37; green, 35; blue, 35 }  ,fill opacity=1 ]  (195.96, 61.36) circle [x radius= 1.34, y radius= 1.34]   ;
     \draw [fill={rgb, 255:red, 37; green, 35; blue, 35 }  ,fill opacity=1 ]  (165.03, 61.36) circle [x radius= 1.34, y radius= 1.34]   ;
     \draw [fill={rgb, 255:red, 37; green, 35; blue, 35 }  ,fill opacity=1 ]  (134.1, 61.36) circle [x radius= 1.34, y radius= 1.34]   ;
     \draw [fill={rgb, 255:red, 37; green, 35; blue, 35 }  ,fill opacity=1 ]  (134.1, 61.36) circle [x radius= 1.34, y radius= 1.34]   ;
     \draw [fill={rgb, 255:red, 37; green, 35; blue, 35 }  ,fill opacity=1 ]  (165.03, 61.36) circle [x radius= 1.34, y radius= 1.34]   ;
     \draw [fill={rgb, 255:red, 37; green, 35; blue, 35 }  ,fill opacity=1 ]  (165.03, 61.36) circle [x radius= 1.34, y radius= 1.34]   ;
     \draw [fill={rgb, 255:red, 37; green, 35; blue, 35 }  ,fill opacity=1 ]  (195.96, 61.36) circle [x radius= 1.34, y radius= 1.34]   ;
     \draw [fill={rgb, 255:red, 37; green, 35; blue, 35 }  ,fill opacity=1 ]  (195.96, 61.36) circle [x radius= 1.34, y radius= 1.34]   ;
     \draw [fill={rgb, 255:red, 37; green, 35; blue, 35 }  ,fill opacity=1 ]  (165.03, 61.36) circle [x radius= 1.34, y radius= 1.34]   ;
     \draw [fill={rgb, 255:red, 37; green, 35; blue, 35 }  ,fill opacity=1 ]  (165.03, 61.36) circle [x radius= 1.34, y radius= 1.34]   ;
     \draw [fill={rgb, 255:red, 37; green, 35; blue, 35 }  ,fill opacity=1 ]  (134.1, 61.36) circle [x radius= 1.34, y radius= 1.34]   ;
     \draw [fill={rgb, 255:red, 37; green, 35; blue, 35 }  ,fill opacity=1 ]  (134.1, 61.36) circle [x radius= 1.34, y radius= 1.34]   ;
     \draw [fill={rgb, 255:red, 37; green, 35; blue, 35 }  ,fill opacity=1 ]  (195.96, 61.36) circle [x radius= 1.34, y radius= 1.34]   ;
     \draw [fill={rgb, 255:red, 37; green, 35; blue, 35 }  ,fill opacity=1 ]  (195.96, 61.36) circle [x radius= 1.34, y radius= 1.34]   ;
     \draw [fill={rgb, 255:red, 37; green, 35; blue, 35 }  ,fill opacity=1 ]  (134.1, 61.36) circle [x radius= 1.34, y radius= 1.34]   ;
     \draw [fill={rgb, 255:red, 37; green, 35; blue, 35 }  ,fill opacity=1 ]  (165.03, 61.36) circle [x radius= 1.34, y radius= 1.34]   ;
     \draw [fill={rgb, 255:red, 37; green, 35; blue, 35 }  ,fill opacity=1 ]  (165.03, 61.36) circle [x radius= 1.34, y radius= 1.34]   ;
     \draw [fill={rgb, 255:red, 37; green, 35; blue, 35 }  ,fill opacity=1 ]  (195.96, 61.36) circle [x radius= 1.34, y radius= 1.34]   ;
     \draw [fill={rgb, 255:red, 37; green, 35; blue, 35 }  ,fill opacity=1 ]  (103.18, 61.36) circle [x radius= 1.34, y radius= 1.34]   ;
     \draw [fill={rgb, 255:red, 37; green, 35; blue, 35 }  ,fill opacity=1 ]  (134.1, 61.36) circle [x radius= 1.34, y radius= 1.34]   ;
     \draw [fill={rgb, 255:red, 37; green, 35; blue, 35 }  ,fill opacity=1 ]  (134.1, 61.36) circle [x radius= 1.34, y radius= 1.34]   ;
     \draw [fill={rgb, 255:red, 37; green, 35; blue, 35 }  ,fill opacity=1 ]  (103.18, 61.36) circle [x radius= 1.34, y radius= 1.34]   ;
     \draw [fill={rgb, 255:red, 37; green, 35; blue, 35 }  ,fill opacity=1 ]  (103.18, 61.36) circle [x radius= 1.34, y radius= 1.34]   ;
     \draw [fill={rgb, 255:red, 37; green, 35; blue, 35 }  ,fill opacity=1 ]  (165.03, 61.36) circle [x radius= 1.34, y radius= 1.34]   ;
     \draw [fill={rgb, 255:red, 37; green, 35; blue, 35 }  ,fill opacity=1 ]  (103.18, 61.36) circle [x radius= 1.34, y radius= 1.34]   ;
     \draw [fill={rgb, 255:red, 37; green, 35; blue, 35 }  ,fill opacity=1 ]  (103.18, 61.36) circle [x radius= 1.34, y radius= 1.34]   ;
     \draw [fill={rgb, 255:red, 37; green, 35; blue, 35 }  ,fill opacity=1 ]  (134.1, 61.36) circle [x radius= 1.34, y radius= 1.34]   ;
     \draw [fill={rgb, 255:red, 37; green, 35; blue, 35 }  ,fill opacity=1 ]  (134.1, 61.36) circle [x radius= 1.34, y radius= 1.34]   ;
     \draw [fill={rgb, 255:red, 37; green, 35; blue, 35 }  ,fill opacity=1 ]  (165.03, 61.36) circle [x radius= 1.34, y radius= 1.34]   ;
     \draw [fill={rgb, 255:red, 37; green, 35; blue, 35 }  ,fill opacity=1 ]  (165.03, 61.36) circle [x radius= 1.34, y radius= 1.34]   ;
     \draw [fill={rgb, 255:red, 37; green, 35; blue, 35 }  ,fill opacity=1 ]  (165.03, 61.36) circle [x radius= 1.34, y radius= 1.34]   ;
     \draw [fill={rgb, 255:red, 37; green, 35; blue, 35 }  ,fill opacity=1 ]  (134.1, 61.36) circle [x radius= 1.34, y radius= 1.34]   ;
     \draw [fill={rgb, 255:red, 37; green, 35; blue, 35 }  ,fill opacity=1 ]  (134.1, 61.36) circle [x radius= 1.34, y radius= 1.34]   ;
     \draw [fill={rgb, 255:red, 37; green, 35; blue, 35 }  ,fill opacity=1 ]  (103.18, 61.36) circle [x radius= 1.34, y radius= 1.34]   ;
     \draw [fill={rgb, 255:red, 37; green, 35; blue, 35 }  ,fill opacity=1 ]  (103.18, 61.36) circle [x radius= 1.34, y radius= 1.34]   ;
     \draw [fill={rgb, 255:red, 37; green, 35; blue, 35 }  ,fill opacity=1 ]  (134.1, 61.36) circle [x radius= 1.34, y radius= 1.34]   ;
     \draw [fill={rgb, 255:red, 37; green, 35; blue, 35 }  ,fill opacity=1 ]  (165.03, 61.36) circle [x radius= 1.34, y radius= 1.34]   ;
     \draw [fill={rgb, 255:red, 37; green, 35; blue, 35 }  ,fill opacity=1 ]  (480.7, 151.08) circle [x radius= 1.34, y radius= 1.34]   ;
     \draw [fill={rgb, 255:red, 37; green, 35; blue, 35 }  ,fill opacity=1 ]  (480.7, 151.08) circle [x radius= 1.34, y radius= 1.34]   ;
     \draw [fill={rgb, 255:red, 37; green, 35; blue, 35 }  ,fill opacity=1 ]  (480.7, 151.08) circle [x radius= 1.34, y radius= 1.34]   ;
     \draw [fill={rgb, 255:red, 37; green, 35; blue, 35 }  ,fill opacity=1 ]  (480.7, 151.08) circle [x radius= 1.34, y radius= 1.34]   ;
     \draw [fill={rgb, 255:red, 37; green, 35; blue, 35 }  ,fill opacity=1 ]  (511.63, 151.08) circle [x radius= 1.34, y radius= 1.34]   ;
     \draw [fill={rgb, 255:red, 37; green, 35; blue, 35 }  ,fill opacity=1 ]  (511.63, 151.08) circle [x radius= 1.34, y radius= 1.34]   ;
     \draw [fill={rgb, 255:red, 37; green, 35; blue, 35 }  ,fill opacity=1 ]  (449.77, 151.08) circle [x radius= 1.34, y radius= 1.34]   ;
     \draw [fill={rgb, 255:red, 37; green, 35; blue, 35 }  ,fill opacity=1 ]  (418.85, 151.08) circle [x radius= 1.34, y radius= 1.34]   ;
     \draw [fill={rgb, 255:red, 37; green, 35; blue, 35 }  ,fill opacity=1 ]  (387.92, 151.08) circle [x radius= 1.34, y radius= 1.34]   ;
     \draw [fill={rgb, 255:red, 37; green, 35; blue, 35 }  ,fill opacity=1 ]  (387.92, 151.08) circle [x radius= 1.34, y radius= 1.34]   ;
     \draw [fill={rgb, 255:red, 37; green, 35; blue, 35 }  ,fill opacity=1 ]  (418.85, 151.08) circle [x radius= 1.34, y radius= 1.34]   ;
     \draw [fill={rgb, 255:red, 37; green, 35; blue, 35 }  ,fill opacity=1 ]  (418.85, 151.08) circle [x radius= 1.34, y radius= 1.34]   ;
     \draw [fill={rgb, 255:red, 37; green, 35; blue, 35 }  ,fill opacity=1 ]  (449.77, 151.08) circle [x radius= 1.34, y radius= 1.34]   ;
     \draw [fill={rgb, 255:red, 37; green, 35; blue, 35 }  ,fill opacity=1 ]  (449.77, 151.08) circle [x radius= 1.34, y radius= 1.34]   ;
     \draw [fill={rgb, 255:red, 37; green, 35; blue, 35 }  ,fill opacity=1 ]  (387.92, 151.08) circle [x radius= 1.34, y radius= 1.34]   ;
     \draw [fill={rgb, 255:red, 37; green, 35; blue, 35 }  ,fill opacity=1 ]  (387.92, 151.08) circle [x radius= 1.34, y radius= 1.34]   ;
     \draw [fill={rgb, 255:red, 37; green, 35; blue, 35 }  ,fill opacity=1 ]  (418.85, 151.08) circle [x radius= 1.34, y radius= 1.34]   ;
     \draw [fill={rgb, 255:red, 37; green, 35; blue, 35 }  ,fill opacity=1 ]  (394.1, 61.36) circle [x radius= 1.34, y radius= 1.34]   ;
     \draw [fill={rgb, 255:red, 37; green, 35; blue, 35 }  ,fill opacity=1 ]  (394.1, 61.36) circle [x radius= 1.34, y radius= 1.34]   ;
     \draw [fill={rgb, 255:red, 37; green, 35; blue, 35 }  ,fill opacity=1 ]  (455.96, 61.36) circle [x radius= 1.34, y radius= 1.34]   ;
     \draw [fill={rgb, 255:red, 37; green, 35; blue, 35 }  ,fill opacity=1 ]  (455.96, 61.36) circle [x radius= 1.34, y radius= 1.34]   ;
     \draw [fill={rgb, 255:red, 37; green, 35; blue, 35 }  ,fill opacity=1 ]  (425.03, 61.36) circle [x radius= 1.34, y radius= 1.34]   ;
     \draw [fill={rgb, 255:red, 37; green, 35; blue, 35 }  ,fill opacity=1 ]  (394.1, 61.36) circle [x radius= 1.34, y radius= 1.34]   ;
     \draw [fill={rgb, 255:red, 37; green, 35; blue, 35 }  ,fill opacity=1 ]  (394.1, 61.36) circle [x radius= 1.34, y radius= 1.34]   ;
     \draw [fill={rgb, 255:red, 37; green, 35; blue, 35 }  ,fill opacity=1 ]  (425.03, 61.36) circle [x radius= 1.34, y radius= 1.34]   ;
     \draw [fill={rgb, 255:red, 37; green, 35; blue, 35 }  ,fill opacity=1 ]  (425.03, 61.36) circle [x radius= 1.34, y radius= 1.34]   ;
     \draw [fill={rgb, 255:red, 37; green, 35; blue, 35 }  ,fill opacity=1 ]  (455.96, 61.36) circle [x radius= 1.34, y radius= 1.34]   ;
     \draw [fill={rgb, 255:red, 37; green, 35; blue, 35 }  ,fill opacity=1 ]  (455.96, 61.36) circle [x radius= 1.34, y radius= 1.34]   ;
     \draw [fill={rgb, 255:red, 37; green, 35; blue, 35 }  ,fill opacity=1 ]  (425.03, 61.36) circle [x radius= 1.34, y radius= 1.34]   ;
     \draw [fill={rgb, 255:red, 37; green, 35; blue, 35 }  ,fill opacity=1 ]  (425.03, 61.36) circle [x radius= 1.34, y radius= 1.34]   ;
     \draw [fill={rgb, 255:red, 37; green, 35; blue, 35 }  ,fill opacity=1 ]  (394.1, 61.36) circle [x radius= 1.34, y radius= 1.34]   ;
     \draw [fill={rgb, 255:red, 37; green, 35; blue, 35 }  ,fill opacity=1 ]  (394.1, 61.36) circle [x radius= 1.34, y radius= 1.34]   ;
     \draw [fill={rgb, 255:red, 37; green, 35; blue, 35 }  ,fill opacity=1 ]  (455.96, 61.36) circle [x radius= 1.34, y radius= 1.34]   ;
     \draw [fill={rgb, 255:red, 37; green, 35; blue, 35 }  ,fill opacity=1 ]  (455.96, 61.36) circle [x radius= 1.34, y radius= 1.34]   ;
     \draw [fill={rgb, 255:red, 37; green, 35; blue, 35 }  ,fill opacity=1 ]  (394.1, 61.36) circle [x radius= 1.34, y radius= 1.34]   ;
     \draw [fill={rgb, 255:red, 37; green, 35; blue, 35 }  ,fill opacity=1 ]  (425.03, 61.36) circle [x radius= 1.34, y radius= 1.34]   ;
     \draw [fill={rgb, 255:red, 37; green, 35; blue, 35 }  ,fill opacity=1 ]  (425.03, 61.36) circle [x radius= 1.34, y radius= 1.34]   ;
     \draw [fill={rgb, 255:red, 37; green, 35; blue, 35 }  ,fill opacity=1 ]  (455.96, 61.36) circle [x radius= 1.34, y radius= 1.34]   ;
     \draw [fill={rgb, 255:red, 37; green, 35; blue, 35 }  ,fill opacity=1 ]  (363.18, 61.36) circle [x radius= 1.34, y radius= 1.34]   ;
     \draw [fill={rgb, 255:red, 37; green, 35; blue, 35 }  ,fill opacity=1 ]  (394.1, 61.36) circle [x radius= 1.34, y radius= 1.34]   ;
     \draw [fill={rgb, 255:red, 37; green, 35; blue, 35 }  ,fill opacity=1 ]  (394.1, 61.36) circle [x radius= 1.34, y radius= 1.34]   ;
     \draw [fill={rgb, 255:red, 37; green, 35; blue, 35 }  ,fill opacity=1 ]  (363.18, 61.36) circle [x radius= 1.34, y radius= 1.34]   ;
     \draw [fill={rgb, 255:red, 37; green, 35; blue, 35 }  ,fill opacity=1 ]  (363.18, 61.36) circle [x radius= 1.34, y radius= 1.34]   ;
     \draw [fill={rgb, 255:red, 37; green, 35; blue, 35 }  ,fill opacity=1 ]  (425.03, 61.36) circle [x radius= 1.34, y radius= 1.34]   ;
     \draw [fill={rgb, 255:red, 37; green, 35; blue, 35 }  ,fill opacity=1 ]  (363.18, 61.36) circle [x radius= 1.34, y radius= 1.34]   ;
     \draw [fill={rgb, 255:red, 37; green, 35; blue, 35 }  ,fill opacity=1 ]  (363.18, 61.36) circle [x radius= 1.34, y radius= 1.34]   ;
     \draw [fill={rgb, 255:red, 37; green, 35; blue, 35 }  ,fill opacity=1 ]  (394.1, 61.36) circle [x radius= 1.34, y radius= 1.34]   ;
     \draw [fill={rgb, 255:red, 37; green, 35; blue, 35 }  ,fill opacity=1 ]  (394.1, 61.36) circle [x radius= 1.34, y radius= 1.34]   ;
     \draw [fill={rgb, 255:red, 37; green, 35; blue, 35 }  ,fill opacity=1 ]  (425.03, 61.36) circle [x radius= 1.34, y radius= 1.34]   ;
     \draw [fill={rgb, 255:red, 37; green, 35; blue, 35 }  ,fill opacity=1 ]  (425.03, 61.36) circle [x radius= 1.34, y radius= 1.34]   ;
     \draw [fill={rgb, 255:red, 37; green, 35; blue, 35 }  ,fill opacity=1 ]  (425.03, 61.36) circle [x radius= 1.34, y radius= 1.34]   ;
     \draw [fill={rgb, 255:red, 37; green, 35; blue, 35 }  ,fill opacity=1 ]  (425.03, 61.36) circle [x radius= 1.34, y radius= 1.34]   ;
     \draw [fill={rgb, 255:red, 37; green, 35; blue, 35 }  ,fill opacity=1 ]  (394.1, 61.36) circle [x radius= 1.34, y radius= 1.34]   ;
     \draw [fill={rgb, 255:red, 37; green, 35; blue, 35 }  ,fill opacity=1 ]  (394.1, 61.36) circle [x radius= 1.34, y radius= 1.34]   ;
     \draw [fill={rgb, 255:red, 37; green, 35; blue, 35 }  ,fill opacity=1 ]  (363.18, 61.36) circle [x radius= 1.34, y radius= 1.34]   ;
     \draw [fill={rgb, 255:red, 37; green, 35; blue, 35 }  ,fill opacity=1 ]  (363.18, 61.36) circle [x radius= 1.34, y radius= 1.34]   ;
     \draw [fill={rgb, 255:red, 37; green, 35; blue, 35 }  ,fill opacity=1 ]  (394.1, 61.36) circle [x radius= 1.34, y radius= 1.34]   ;
     \draw [fill={rgb, 255:red, 37; green, 35; blue, 35 }  ,fill opacity=1 ]  (425.03, 61.36) circle [x radius= 1.34, y radius= 1.34]   ;
     \draw [fill={rgb, 255:red, 37; green, 35; blue, 35 }  ,fill opacity=1 ]  (425.03, 61.36) circle [x radius= 1.34, y radius= 1.34]   ;
     \draw [fill={rgb, 255:red, 37; green, 35; blue, 35 }  ,fill opacity=1 ]  (326.06, 151.08) circle [x radius= 1.34, y radius= 1.34]   ;
     \draw [fill={rgb, 255:red, 37; green, 35; blue, 35 }  ,fill opacity=1 ]  (326.06, 151.08) circle [x radius= 1.34, y radius= 1.34]   ;
     \draw [fill={rgb, 255:red, 37; green, 35; blue, 35 }  ,fill opacity=1 ]  (326.06, 151.08) circle [x radius= 1.34, y radius= 1.34]   ;
     \draw [fill={rgb, 255:red, 37; green, 35; blue, 35 }  ,fill opacity=1 ]  (356.99, 151.08) circle [x radius= 1.34, y radius= 1.34]   ;
     \draw [fill={rgb, 255:red, 37; green, 35; blue, 35 }  ,fill opacity=1 ]  (356.99, 151.08) circle [x radius= 1.34, y radius= 1.34]   ;
     \draw [fill={rgb, 255:red, 37; green, 35; blue, 35 }  ,fill opacity=1 ]  (326.06, 151.08) circle [x radius= 1.34, y radius= 1.34]   ;
     \draw [fill={rgb, 255:red, 37; green, 35; blue, 35 }  ,fill opacity=1 ]  (356.99, 151.08) circle [x radius= 1.34, y radius= 1.34]   ;
     \end{tikzpicture}
\end{center}

\begin{lemma}\label{property for H}
Let $k_1\geq 2$, $k_2\geq 1$, $k_1\geq k_2$, and $k=k_1+k_2+1$. If $s\leq k-1$ then $H_k(s)$ and $H_k(M_2)$ satisfy the following.\\
(\romannumeral1) After deleting any vertex of $H_k(1)$ or $H_k(M_2)$, the resulting graph contains $P_{2k_1+1}\cup P_{2k_2+1}$ as a subgraph.\\
(\romannumeral2) For $s\leq k-2$, each edge of $H_k(s)$ ($H_k(M_2)$) is contained in a cycle of length $2k$.\\
(\romannumeral3) After adding any new edge inside $A\cup B\cup D$ of $H_k(s)$ ($H_k(M_2)$), the obtained graph is Hamiltonian.\\
($\romannumeral4$) For $x\in C$ and $y\in D$, if $s\leq k-2$ then there is a path on $2k$ vertices
in $H_k(s)$ ($H_k(M_2)$) starting from $x$ and ending in $y$.
\end{lemma}
\begin{proof}
For $H_k(1)$, if we delete any vertex in $A\cup B\cup D$, then it is easy to
check that the resulting graph contains $P_{2k_1+1}\cup P_{2k_2+1}\subseteq P_{2k}$
as a subgraph. If we delete $x_{2i}$ for $i\geq k_2+1$, then $H_k(1)$ contains
$P_{2k_2+1}=x_1P x_{2k_2+1}$ and $P_{2k_1+1}=x_{2i+1}Px_{2k+1}x_{2k_2+2}Px_{2i-1}$
as a subgraph. If $i\leq k_2$, similarly, $H_k(1)$ contains $P_{2k_1+1}\cup P_{2k_2+1}$
as a subgraph.

For $H_k(M_2)$, we have $|D|=k-2\geq 2$. If we delete any vertex in $A\cup B\cup D$
of $H$, then $H_k(M_2)$ contains $P_{2k_1+1}\cup P_{2k_2+1}\subseteq P_{2k}$ as a subgraph.
If we delete $x_{2i+1}\in C$ for $1\leq i\leq k-1 $,
then $H_k(M_2)[D\cup C\setminus \{x_{2i+1}\}]$ contains a path through $C$ and
$D$ on $2k-4$ vertices which starts at $C$ and ends in $D$. Hence
$H_k(M_2)[D\cup C \setminus \{x_{2i+1}\}]$ contains a path on $2k_2+1$ vertices starting and ending in $D$ and a path on $2k_1-3$ vertices starting and ending in $C$. Since $x_1,x_2,x_{2k},x_{2k+1}$ are joint to every vertex of $C$, $H_k(M_2)$ contains $P_{2k_1+1}\cup P_{2k_2+1}$ as a subgraph. Thus we finish the proof of $(\romannumeral1)$.

First we show that each edge of $H_k(s)$ is contained in a $C_{2k}$ for $s\leq k-2$.
Since $s\leq k-2$, $|D|=k-s\geq 2$. For the edges inside $A$ and $B$, without
loss of generality, $x_1x_2$ is contained in $x_1x_2Px_{2k-s-1}x_{2k+1}Px_{2k-s+1}x_1$.
For the edges between $A\cup B$ and $C$, without loss of generality, $x_1x_{s+2i+1}$
is contained in $x_1Px_{s+2i-1}x_{2k+1}Px_{s+2i+1}x_1$ for $i=1,\ldots, k-s$. In
addition, $x_1x_{s+2i+1}$ is contained in
$x_1x_{s+1}Px_{2k-s-1}x_{2k+1}Px_{2k-s+1}x_k\ldots x_1$ for $i=0$. For the edges
between $C$ and $D$, each edge is contained in either  $x_{s+3}Px_{2k+1}x_{s+1}Px_1 x_{s+3}$
or $x_{2k-s-1}Px_1x_{2k-s+1}Px_{2k+1}x_{2k-s-1}$.

For $H_k(M_2)$, we have $|D|=k-2\geq 2$. Since there is only one missing edge
between $C$ and $D$, any edge between $C$ and $D$ is contained in a Hamilton path in
$H_k(M_2)[C\cup D]$. Thus, each edge of $H_k(M_2)$ is contained in a copy of $H_k(2)$ in
$H_k(M_2)$, the result follows from the previous proof. Hence we finish the proof of $(\romannumeral2)$.

If we adding an edge between $A$ and $B$, the obtained graph is obviously Hamiltonian.
If we add an edge $x_{s+2i}x_{s+2j}$ for $1\leq i<j\leq k-s+1$ inside $D$, then
$x_{s+2i}P x_1x_{s+2i+1}Px_{s+2j-1}x_{2k+1}Px_{s+2j}x_{s+2i}$
is a Hamilton cycle. If we add an edge $x_ix_{s+2j}$ for $1\leq i\leq s$ and
$1\leq j\leq k-s$ between $A$ and $D$, then
$x_iPx_1x_{i+1}Px_{s+2j-1}x_{2k+1}P$ $x_{s+2j}x_i$ is a Hamilton cycle. If we
add an edge $x_ix_{s+2j}$ for $2k-s+2\leq i\leq 2k+1$ and $1\leq j\leq k-s$ between $B$ and $D$, then
$x_iPx_mx_{i-1}Px_{s+2j+1}x_{1}Px_{s+2j}x_i$ is a Hamilton cycle. The proof of
$(\romannumeral3)$ is complete.

For $x=x_{s+2i-1}$ and $y=x_{s+2j}$, where $s<k-1$. If $i\leq j-1$, then
$x_{s+2i-1}Px_1x_{s+2i+1}P$
$x_{s+2j-1}x_{2k+1}Px_{s+2j}$ is an $(x,y)$-path on $2k$ vertices. If
$i\geq j+2$, then  $x_{s+2j}Px_1x_{s+2j+1}P$
$x_{s+2i-1}Px_{2k+1}x_{s+2i-3}$ is an $(x,y)$-path on $2k$ vertices.
If $i=j>1$, then  $x_{s+2i-1}Px_{s+3}x_{1}P$
$x_{s+1}x_{2k+1}Px_{s+2j}$ is an $(x,y)$-path on $2k$ vertices. If
$i=j=1$, then  $x_{s+1}Px_1x_{s+5}Px_{2k+1}$
$x_{s+3}x_{s+2}$ is an  $(x,y)$-path on $2k$ vertices.
if $i=j+1=2$, then  $x_{s+3}x_{2k+1}Px_{s+5}x_{1}Px_{s+2}$ is an
$(x,y)$-path on $2k$ vertices. If $i=j+1>2$, then  $x_{s+2i-1}Px_{2k+1}x_{s+1}Px_{1}x_{s+3}Px_{s+2j}$
is an $(x,y)$-path on $2k$ vertices.
We finish the proof of $(\romannumeral4)$.
\end{proof}

\begin{lemma}\label{H(2k+1)}
Let $k_1\geq k_2\geq 2$ and $k=k_1+k_2+1$. Then $H_k(P_3)$ satisfies the following.\\
(\romannumeral1) After deleting any vertex of $H_k(P_3)$, the obtained graph contains $P_{2k_1+1}\cup P_{2k_2+1}$ as a subgraph.\\
(\romannumeral2) Each edge of $H_k(P_3)$ except for the edges between $a_2$ and $B$ is contained
in a cycle of length $2k$. Moreover, there is a path on $2k$ vertices starting from $a_1$ and ending in $a_3$.\\
(\romannumeral3) After adding any new edge except for $a_1a_3$ inside $(A\setminus\{a_2\})\cup C$
of $H_k(P_3)$, the obtained graphs are both Hamiltonian. Moreover, for $z\in C$, there is a path
on $2k$ vertices starting from $a_2$ and ending in $z$.\\
(\romannumeral4) For $x\in B\cup \{a_2\}$ and $y\in B\cup \{a_2\}$, there is a path on $2k-1$
vertices in $H_k(P_3)$ starting from $x$ and ending in $y$. Moreover, if $a_1$ is joint to $a_3$,
then there is a path on $2k$ vertices in $H_k(P_3)$ starting from $B$ and ending in $a_2$.
\end{lemma}
\begin{proof}
Note that there is a path on $2k+1$ vertices from any vertex in $C$ to any vertex in $\{a_1,a_3\}$.
Thus, if we delete any vertex in $\{a_1,a_3\}\cup C$, then the obtained graph contains
$P_{2k_1+1}\cup P_{2k_2+1}\subseteq P_{2k}$ as a subgraph.
If we delete $a_2\in A$, since $H_k(P_3)[B\cup C]$ contains a path $P_0$ on $2k_1$
vertices starting from $C$ ending in $B$ and a path $P_1$ on $2k_2$ vertices starting
from $C$ ending in $B$, then $a_1P_0$ is a path on $2k_1+1$ vertices and $a_3P_1$ is a path on $2k_2+1$ vertices.
Hence, $H_k(P_3)\setminus \{a_{2}\}$ contains $P_{2k_1+1}\cup P_{2k_2+1}$ as a subgraph.
If we delete $b_{i}\in B$ for $1\leq i\leq k-1$, then $H_k(P_3)[B\setminus \{b_{i}\}\cup C]$ contains
a path on $2k-3$ vertices from $C$ to $B$.
Hence $H_k(P_3)[B\setminus\{b_{i}\}\cup C]$ contains a path on $2k_1+1$ vertices starting and
ending in $C$ and a path $P_2$ on $2k_2-2$ vertices starting in $C$ and ending in $B$. Since $a_3$ is joint to every vertex of $B$, $a_1a_2a_3P_2$ is a path on $2k_2+1$.
Hence $H_k(P_3)\setminus \{b_{i}\}$ contains $P_{2k_1+1}\cup P_{2k_2+1}$ as a subgraph and
we finish the proof of $(\romannumeral1)$.

Note that there is a cycle starting  $a_1a_2a_3$, through $B$ and $C$ alternately, and ending at $a_1$.
By symmetry, the edges between $B$ and $\{a_1,a_3\}\cup C$ are contained in a cycle of length $2k$.
We finish the proof of $(\romannumeral2)$.

For $H_k(P_3)$, note that there is a path on $2k+1$ vertices from any vertex in $C$ to any vertex in $A\setminus\{a_2\}$.
Thus, if we add an edge between $A\setminus\{a_2\}$ and $C$, the obtained graph is obviously Hamiltonian.
Without loss of generality, if we add an edge $c_1c_2$ inside $C$, then
$c_1c_2b_2c_3\ldots c_{k-1}b_{k-1}a_1a_2a_3b_1c_1$ is a Hamilton cycle.
Additionally, for $z\in C$, there is a path $P_3$ on $2k-2$ vertices in
$H_k(P_3)[B\cup C]$ starting from $B$ and ending in $z$.
Then $a_2a_1P_3z$ is a path on $2k$ vertices starting at $a_2$ and ending in $z$.
The proof of $(\romannumeral3)$ is complete.

For $x=b_{i}$ and $y=b_{j}$ ($i\neq j$), let $1\leq z\leq k-1$, $z\neq i$ and $z\neq j$.
Then $H_k(P_3)[C\cup B\setminus \{b_j\}]$ contains a path $P_4$ on $2k-5$ vertices from $b_i$ to $b_z$.
Then $b_iP_4b_za_1a_2a_3b_j$ is a path on $2k-1$ vertices starting at $x$ and ending in $y$.
For $x=b_{i}$ and $y=a_2$, $H_k(P_3)[B\cup C]$ contains a path $P_5$ on $2k-3$ vertices from $B$ to $b_i$.
Then, $b_iP_5a_3a_2$ is a path on $2k-1$ vertices from $x$ and ending in $y$.
If $a_1$ is joint to $a_3$, then $b_iP_5a_3a_1a_2$ is a path on $2k$ vertices from $B$ and ending in $a_2$.
We finish the proof of $(\romannumeral4)$.
\end{proof}

\begin{proposition}\label{eq-1}
For $k_1\geq k_2\geq 2$ and $2k_1+2k_2+2\leq |C|\leq n$, we have
\begin{equation*}
h(|C|,2k_1+2k_2+1,1)+\mathrm{ex}(n-|C|,P_{2k_2+1})<c(n,2k_1+1,2k_2+1).
\end{equation*}

\end{proposition}
\begin{proof}
If $|C|\geq 2k_1+2k_2+2$, then
\begin{align}\label{adddddd}
&h(|C|,2k_1+2k_2+1,1)+\mathrm{ex}(n-|C|,P_{2k_2+1})\\\nonumber
=& {2k_1+2k_2 \choose 2}  +  |C|-2k_1-2k_2 +\mathrm{ex}(n-|C|,P_{2k_2+1})   \\\nonumber
=& {2k_1+2k_2+1 \choose 2}  +  |C|-4k_1-4k_2 +\mathrm{ex}(n-|C|,P_{2k_2+1})   \\\nonumber
=& {2k_1+2k_2+1 \choose 2}  -(2k_1+2k_2-1)  +(|C|-2k_1-2k_2-1) +\mathrm{ex}(n-|C|,P_{2k_2+1}).
\end{align}

Suppose $x=|C|-2k_1-2k_2-1$. We will prove that
$$
\mathrm{ex}(n-|C|,P_{2k_2+1})+x-(2k_1-2k_2-1)<\mathrm{ex}(n-|C|+x,P_{2k_2+1}).
$$
If $0\leq x\leq 2k_1+2k_2-1$, the conclusion is clearly true.

If $x\geq 2k_1+2k_2-1\geq 7$, since $k_1\geq k_2\geq 2$, we have $x\geq 2k_2+1$.
By Theorem~\ref{extrem}, suppose $n-|C|=2sk_2+r$ and $0\leq r\leq 2k_2-1$, then
$$
\mathrm{ex}(n-|C|,P_{2k_2+1})=s{2k_2\choose 2}+{r\choose 2}.
$$
Suppose $x+r=2s^\prime k_2+r_0$, where $0\leq r_0\leq 2k_2-1$ and $s^\prime \geq 1$. Then by Theorem~\ref{extrem},
$\mathrm{ex}(n-|C|+x,P_{2k_2+1})=(s+s^\prime){2k_2\choose 2}+{r_0\choose 2}$, we have
\begin{align*}
               &\mathrm{ex}(n-|C|+x,P_{2k_2+1})-(\mathrm{ex}(n-|C|,P_{2k_2+1})+x)\\\nonumber
               =&\left((s+s^\prime){2k_2\choose 2}+{r_0\choose 2}\right)-\left(s{2k_2\choose 2}+{r\choose 2}+x\right)\\\nonumber
                =&s^\prime{2k_2\choose 2}+{r_0\choose 2}-{r\choose 2}-(2s^\prime k_2+r_0-r)\\\nonumber
               =& \frac{1}{2}(2s^\prime k_2(2k_2-3)+r_0(r_0-3)-r(r-3))\\\nonumber
               \geq& \frac{1}{2}(2s^\prime k_2(2k_2-3)-r(r-3))\\\nonumber
              \geq& \frac{1}{2}(2s^\prime k_2(2k_2-3)-(2k_2-1)(2k_2-4)) \\\nonumber
               >&0,
               \end{align*}
we also have
$$
\mathrm{ex}(n-|C|,P_{2k_2+1})+x-(2k_1-2k_2-1)<\mathrm{ex}(n-|C|+x,P_{2k_2+1}).
$$

By (\ref{adddddd}), we have
\begin{align*}
&h(|C|,2k_1+2k_2+1,1)+\mathrm{ex}(n-|C|,P_{2k_2+1})\\\nonumber
<& {2k_1+2k_2+1 \choose 2}  +   \mathrm{ex}(n-2k_1-2k_2-1,P_{2k_2+1})   \\\nonumber
=&c(n,2k_1+1,2k_2+1),
\end{align*}
where the strict inequality holds by the fact $k_2\geq 2$ and Theorem~\ref{extrem}.
\end{proof}
\begin{proposition} \label{compa}
For $k_1\geq k_2\geq 2$, if $2k_1+2k_2+2\leq |C|\leq n$ then
$$h(|C|,2k_1+2k_2,k_1+k_2-1)+\mathrm{ex}(n-|C|,P_{2k_2+1})\leq h(n,2k_1+2k_2,k_1+k_2-1).$$
\end{proposition}
\begin{proof}
By Theorem~\ref{THM: path}, we can suppose $n-|C|=s(2k_2)+r$ and $0\leq r\leq 2k_2-1$.
Then,
$$
\mathrm{ex}(n-|C|,P_{2k_2+1})=s{2k_2\choose 2}+{r\choose 2}.
$$  If $2k_1+2k_2+2\leq |C|\leq n$, then
\begin{align*}
&h(|C|,2k_1+2k_2,k_1+k_2-1)+\mathrm{ex}(n-|C|,P_{2k_2+1})\\\nonumber
\leq& {k_1+k_2+1 \choose 2}  +  (|C|-(k_1+k_2+1))(k_1+k_2-1) +\frac{1}{2}(2k_2-1)(n-|C|).
\end{align*}

Since $k_1\geq k_2\geq 2$, we have $k_1+k_2-1>(2k_2-1)/2$, and so
\begin{align*}
&h(|C|,2k_1+2k_2,k_1+k_2-1)+\mathrm{ex}(n-|C|,P_{2k_2+1})\\\nonumber
\leq & {k_1+k_2+1 \choose 2}  +  (n-(k_1+k_2+1))(k_1+k_2-1)\\\nonumber
=& h(n,2k_1+2k_2,k_1+k_2-1).
\end{align*}

We finish the proof of Proposition~\ref{compa}.
\end{proof}

\begin{proposition}\label{claim9}
For $k\geq 5$ and $n>2k+1$, we have
$$(k-2)(n-2)/2+2n-3<h(n,2k,k-1).$$
\end{proposition}
\begin{proof}
Since $k\geq 5$ and $n>2k+1$, we have
\begin{align*}
&h(n,2k,k-1)-\left(\frac{(k-2)(n-2)}{2}+2n-3\right)\\
=&{k+1\choose 2}+(k-1)(n-k-1)-\left(\frac{(k-2)(n-2)}{2}+2n-3\right)\\
=&-\frac{k^2}{2}+\frac{3}{2}k+2+(k/2-2)n\\
>&-\frac{k^2}{2}+\frac{3}{2}k+2+(k/2-2)(2k+1)\\
=&\frac{k^2}{2}-2k>0.
\end{align*}
The proof is complete.
\end{proof}

\section{Proof of Theorem \ref{turan two odd path} (assuming Lemma \ref{Lemma:Stablity-Kopylov2})}\label{Sec:4}
In this section, we will prove Theorem \ref{turan two odd path} assuming Lemma \ref{Lemma:Stablity-Kopylov2}.

\begin{theorem}\label{3.2}
Let $G$ be a connected graph on $n$ vertices, $k_1\geq k_2\geq 2$ and $k=k_1+k_2+1$. If $G$ does not contain $P_{2k_1+1}\cup P_{2k_2+1}$ as a subgraph, then
$$
e(G)\leq \max\{h(n,2k_1+2k_2+1,1),h(n,2k_1+2k_2,k_1+k_2-1)\}.
$$
\end{theorem}
\begin{proof}
Let $G$ be a connected graph without containing $P_{2k_1+1}\cup P_{2k_2+1}$.
Let $G^*$ be the graph obtained from $G$ by adding a new vertex $x$ and
joining all edges between $x$ and $V(G)$. Hence $G^*$ is 2-connected.
If $e(G^*)>\max\{h(n+1,2k,k-1),h(n+1,2k+1,2)\}$, then by
Lemma~\ref{Lemma:Stablity-Kopylov2}, $G^*$ contains either $H_k(1)$,
$H_k(M_2)$, or $H_k(P_3)$ as a subgraph, or a cycle with length at least $2k+1$.
In the former case, by Lemma~\ref{property for H}(\romannumeral1) and
Lemma~\ref{H(2k+1)}(\romannumeral1), $G$ contains $P_{2k_1+1}\cup P_{2k_2+1}$
as a subgraph. For the later case, $G$ contains $P_{2k_1+1}\cup P_{2k_2+1}\subseteq P_{2k}$
as a subgraph.
Thus $e(G^*)\leq\max\{h(n+1,2k,k-1),h(n+1,2k+1,2)\}$, and so
\begin{align*}
     e(G)&=e(G^*)-n\\
     &\leq\max\{h(n+1,2k,k-1),h(n+1,2k+1,2)\}-n\\
     &=\max\{h(n+1,2k+1,2)-n,h(n+1,2k,k-1)-n\}\\
     &=\max\{h(n,2k-1,1),h(n,2k-2,k-2)\}.
 \end{align*}
The proof is complete.
\end{proof}
Let
\begin{equation}\label{fx}
f(n,k_1,k_2)=\max\{h(n,2k_1+2k_2+1,1),h(n,2k_1+2k_2,k_1+k_2-1)\},
\end{equation}
and
\begin{equation}\label{gx}
g(n,k_1,k_2)=\max\left\{c(n,2k_1+1,2k_2+1),\mathrm{ex}(n,P_{2k_1+1}), h(n,2k_1+2k_2,k_1+k_2-1)\right\}.
\end{equation}
Now we will show that Lemma \ref{Lemma:Stablity-Kopylov2} and Theorem \ref{3.2} imply Theorem \ref{turan two odd path}.

\medskip

\noindent \textbf{Proof of Theorem \ref{turan two odd path}.}
Suppose that $G$ is an extremal graph for $P_{2k_1+1}\cup P_{2k_2+1}$.
By Theorem~\ref{Y}, we may suppose that $k_1\geq k_2\geq 2$.
Note that the graphs in $\mathcal{C}(n,2k_1+1,2k_2+1)$, $\mathrm{Ex}(n,P_{2k_1+1})$, and the graph $H(n,2k_1+2k_2,k_1+k_2-1)$ do not contain $P_{2k_1+1}\cup P_{2k_2+1}$ as a subgraph.
Thus $e(G)\geq g(n,k_1,k_2)$.
We may suppose for a contradiction that
\begin{equation}\label{assuming condition}
e(G)> g(n,k_1,k_2).
\end{equation}

We divide the proof into the following two cases.

\medskip

\noindent{\bf Case 1.} $G$ is connected.

\medskip

By (\ref{assuming condition}), we have
\begin{align*}
e(G)&>\max\left\{c(n,2k_1+1,2k_2+1),\mathrm{ex}(n,P_{2k_1+1}),h(n,2k_1+2k_2,k_1+k_2-1)\right\} \\
&\geq \max\{c(n,2k_1+1,2k_2+1),h(n,2k_1+2k_2,k_1+k_2-1)\}.
\end{align*}

Since $n\geq 2k_1+2k_2+2\geq 7$, we apply Proposition~\ref{eq-1} with $n=|C|$, and have
\begin{equation*}
e(G)\geq \max\{h(n,2k_1+2k_2+1,1),h(n,2k_1+2k_2,k_1+k_2-1)\}.
\end{equation*}

Thus, $G$ contains $P_{2k_1+1}\cup P_{2k_2+1}$ as a subgraph by Theorem \ref{3.2}, a contradiction.

\medskip

\noindent{\bf Case 2.} $G$ is not connected.

\medskip

If $G$ is not connected, since $e(G)>\mathrm{ex}(n,P_{2k_1+1})$, $G$ contains $P_{2k_1+1}$ as a subgraph.
Let $C$ be the component of $G$ contains $P_{2k_1+1}$ as a subgraph. Then, $G-C$ is $P_{2k_2+1}$-free.
If $|C|\leq 2k_1+2k_2+1$, then
\begin{equation*}
e(G)\leq {|C| \choose 2}+\mathrm{ex}(n-|C|,P_{2k_2+1}).
\end{equation*}

By Theorem~\ref{extrem}, we have
\begin{align*}
e(G)&\leq {2k_1+2k_2+1 \choose 2}+\mathrm{ex}(n-2k_1-2k_2-1,P_{2k_2+1})\\
&=c(n, 2k_1+1, 2k_2+1)\\
&\leq  g(n,k_1,k_2),
\end{align*}
contradicting (\ref{fx}).
If $|C|\geq 2k_1+2k_2+2$, then by Theorem \ref{3.2}, we have
\begin{align*}
          e(G)&\leq f(|C|,k_1,k_2)+\mathrm{ex}(n-|C|,P_{2k_2+1})\\
            &=\max\{h(|C|,2k_1+2k_2+1,1),h(|C|,2k_1+2k_2,k_1+k_2-1)\}+\mathrm{ex}(n-|C|,P_{2k_2+1}).
\end{align*}

By Propositions~\ref{eq-1}, \ref{compa}, and Theorem~\ref{THM: path}, we have
\begin{align*}
e(G)&\leq   \max\left\{c(n,2k_1+1,2k_2+1),h(n,2k_1+2k_2,k_1+k_2-1)\right\}\\
&\leq g(n,k_1,k_2),
\end{align*}
contradicting (\ref{assuming condition}).
The proof of Theorem \ref{turan two odd path} is complete.\hfill $\square$

\section{Proof of Lemma \ref{Lemma:Stablity-Kopylov2}}\label{Sec:5}
We need the following result of Fan \cite{Fan}.
\begin{theorem}[Fan \cite{Fan}]\label{Fan}
Let $G$ be an $n$-vertex 2-connected graph and $ab$ be an edge in $G$.
If the longest path starting from $a$ and ending at $b$ in $G$ has
at most $r$ vertices, then $e(G)\leq (r-3)(n-2)/2+2n-3$.
Moreover, equality holds if and only if $G-\{a,b\}$ consists of
vertex-disjoint union of copies of $K_{r-2}$.
\end{theorem}

Before proving Lemma \ref{Lemma:Stablity-Kopylov2}, we state
the following definition which is important in many proofs related with
problems on circumference.
\begin{mydef}
For an integer $\alpha$ and a graph $G$, the $\alpha$-disintegration
of $G$, denoted by $H(G,\alpha)$, is the graph obtained from $G$ by recursively
deleting vertices of degree at most $\alpha$ until that the resulting graph has no such vertex.
\end{mydef}

\noindent{\bf Proof of Lemma \ref{Lemma:Stablity-Kopylov2}.}
Let $k\geq 5$ and $G$ be the 2-connected $\mathcal{F}$-free graph with maximal edges satisfying (\ref{bound for 2-connected}).
Thus, if $x$ is not adjacent to $y$, then $G+xy$ contains a copy of  $F\in \mathcal{F}$ as a subgraph.
Apply to the graph $G$ the process of $\alpha$-disintegration with $\alpha=k-1$ and let $H=H(G,k-1)$ and $|V(H)|=\ell$.
\begin{claim}\label{complete}
$\ell\geq k+2$.
\end{claim}
\begin{proof}
If $\ell\leq k+1$ then
\begin{align*}
e(G)&\leq {\ell \choose 2}+ (n-\ell)(k-1)\\
&< {2k-(k-1)\choose 2}+(n-2k+(k-1))(k-1)\\
&= h(n,2k,k-1)\\
&\leq \max\{h(n,2k,k-1),h(n,2k+1,2)\},
\end{align*}
a contradiction to (\ref{bound for 2-connected}).
Therefore $|V(H)|\geq k+2$.
\end{proof}
     \begin{claim}\label{Hcomplete}
          $H$ is a complete graph.
     \end{claim}
\begin{proof}
Suppose not. There exist vertices $x$ and $y$ in $H$ such that $x$ is not joint to $y$.
By the maximality of $G$, if we add the edge $xy$ to $G$, the obtained graph $G^*=G+xy$ contains a copy of  $F\in \mathcal{F}$.
We divide our proof of Claim~\ref{Hcomplete} into the following three cases.

\medskip

\noindent{\bf Case 1.} $G^*$ contains a cycle with length at least $2k+1$.
\medskip

Since $G^*$ contains a cycle with length at least $2k+1$, we can choose a path $P=xPy$
of maximum length such that it starts and ends in $H$.
It follows from the maximality of $P$ that  $N_H(x)\subseteq V(P)$ and $N_H(y)\subseteq V(P)$.
If $d_P(x)+d_P(y)\geq d_H(x)+d_H(y)\geq  2k+1$, since $G$ is 2-connected, it follows
from Lemma \ref{posa lemma} that $c(G)\geq 2k+1$, a contradiction.
Hence $N_H(x)=N_P(x)$, $N_H(y)=N_P(y)$, and $d_H(x)=d_H(y)=k$, implying that  $G[V(P)]$
satisfies the condition of Lemma~\ref{extend posa lemma} (applying Lemma \ref{posa lemma} for
vertices $u\in N_H(x)\cup N_H(y)$ and $v\in N_H(x)\cup N_H(y)$ such that there is a path
starting from $u$ and ending at $v$ in $G[V(P)]$ on $|V(P)|$ vertices).
Applying Lemma~\ref{extend posa lemma}, we have $H_k(s)\subseteq G[V(P)]$
for some $s$ when the length of $P$ is $2k$ or $F_k(t)\subseteq G[V(P)]$ for some $t$ when the length of $P$ is at least $2k+1$.

\medskip

\noindent{\bf Subcase 1.1.} The length of $P$ is $2k$, i.e., $H_k(s) \subseteq G[V(P)]$ for some $s$.

\medskip

If $s=1$, then $G$ contains $H_k(1)$ as a subgraph, a contradiction.
Thus, we may assume that $2\leq s\leq k-1$.
\textcolor{blue}{Let $A$, $B$, $C$ and $D$ be the vertex sets of $V(H_k(s))$ after the definition of $H_k(s)$.}
By the good properties of $H_k(s)$, we will show the following claim.
\begin{claim4}\label{good property}
Each vertex of $G-A\cup B \cup C$ can only be adjacent to the vertices in $C$.
\end{claim4}
\begin{proof}
Since $c(G)\leq 2k$, by Lemma~\ref{property for H}(\romannumeral3),
each vertex in $D$ can only be adjacent to the vertices in $C$.
Let $u$ and $v$ be two vertices of $A\cup B \cup C$ which are not both in $C$.

If $s\leq k-2$, by Lemma~\ref{property for H}(\romannumeral2), (\romannumeral3),
and (\romannumeral4), there exists a path on at least $2k$ vertices starting
from $u$ and ending at $v$ in $H_k(s)$.
Thus, each vertex of $G-V(H_k(s))$ can only be adjacent to the vertices in $C$,
otherwise since $G$ is 2-connected, we can easily find a cycle on at least $2k+1$
vertices, a contradiction.

Let $s=k-1$. Then $A\cup B\cup C\subseteq V(H)$, $C=\{x_{k}, x_{k+2}\}$ and $D=\{x_{k+1}\}$.
If $u$ and $v$ are two vertices of $A\cup B$, then w.l.o.g., we can suppose
$u=x_1$ and $v=x_{2k+1}$ ($u\in A$ and $v\in B$), or $u=x_1$ and $v=x_{k-1}$
($\{u,v\}\subseteq A$ or $\{u,v\}\subseteq B$). Then $x_1\ldots x_{2k+1}$
and $x_1\ldots x_{k-2}x_{k}x_{k+3}\ldots x_{2k+1}x_{k+2}x_{k-1}$ are paths
on at least $2k$ vertices from $u$ to $v$ respectively.
If $u\in C$ (w.l.o.g. $u=x_{k}$) and $v\notin C$ (w.l.o.g $v=x_1$), then
$x_1\ldots x_{k-1}x_{k+2}x_{k+3}\ldots x_{2k+1}x_k$ is a path on $2k$
vertices from $u$ to $v$.
We also come to the conclusion that there exists a path on at least $2k$
vertices starting from $u$ and ending at $v$ in $H_k(s)$.
Similarly, each vertex of $G-V(H_k(s))$ also can only be adjacent to the
vertices in $C$; otherwise since $G$ is 2-connected, we have $c(G)\geq 2k+1$,
a contradiction.
The proof of Claim~\ref{good property} is complete.
\end{proof}

From  Claim~\ref{good property}, we have
\begin{align*}
e(G)&=e(G[V(P)\setminus D])+e(G[V(P)\setminus D,V(G)\setminus (A\cup B\cup C)]) \\
&\leq \frac{1}{2}( (k-s+1)(k+s)+2sk)  +(n-k-s-1)(k-1)       \\
&= {k-1 \choose 2}+(n-k+1)(k-1)+2  -\frac{1}{2}(s^2-3s+2)\\
&= h(n,2k,k-1)+1  -\frac{1}{2}(s^2-3s+2).
\end{align*}
Thus, by (\ref{bound for 2-connected}), we have $s=2$ and $e(G)=h(n,2k,k-1)+1$.
Moreover, each vertex in $G-V(P)$ has degree $k-1$, implying
\begin{equation}\label{H'}
e(G[V(P)])=h(2k+1,2k,k-1)+1.
\end{equation}
Note that there are no edges between $A\cup B$ and $D$ and no edges inside $D$ (otherwise
there is a cycle on $2k+1$ vertices by Lemma~\ref{property for H}(\romannumeral3), a contradiction).
By (\ref{H'}), we have $e(A\cup B\cup D, C)=(k+2)(k-1)$.
Thus, $G[V(P)]$ contains a copy of $H_k(2)$ with all edges between $C$ and $D$, implying $G$ contains a copy of $H_k(M_2)$.
\medskip

\noindent{\bf Subcase 1.2.} The length of $P$ is $2k+t-1\geq 2k+1$, i.e.,  $F_k(t) \subseteq G[V(P)]$ for some $t$.

\medskip

Let $P_{2k+t}=x_1\ldots x_{2k+t}$ be a path on $2k +t$ vertices with $t\geq 2$. Let $A=\{x_1,\ldots,x_{k-1}\}$,
$B=\{x_{2k+t-k+2},\ldots,x_{2k+t}\}$, and $C=\{x_{k+1},\ldots,x_{2k+t-k}\}$.
Since $x_1$ is joint to $x_{2k+t-k+1}$, we have $t\leq k-1$; as otherwise,
$x_1P_{2k+t}x_{2k+t-k+1}x_1$ is a cycle on $k+t+1\geq 2k+1$ vertices, a contraction.
\begin{claim4}\label{AB}
$A$ and $B$ are components in $G-\{x_k,x_{2k+t-k+1}\}$.
\end{claim4}
\begin{proof}
Suppose that $p$ is a vertex in $G-A\cup B\cup \{x_k,x_{2k+t-k+1}\}$ adjacent to $y$ in $A\cup B$.
Since $G$ is 2-connected, there must be a path $Q_1$ connecting $p$ to $A\cup B\cup \{x_k,x_{2k+t-k+1}\}$.
Then, we can easily find a cycle on at least $2k+1$ vertices, since there
is an $(y,y^\prime)$-path on $2k$ vertices in $G[V(A\cup B\cup \{x_k,x_{2k+t-k+1}\})]$ for any $y^\prime \neq y$, a contradiction.
The proof the claim is complete.
\end{proof}
\begin{claim4}
There is no $(x_k,x_{2k+t-k+1})$-path on at least $k+2$ vertices in $G$.
\end{claim4}
\begin{proof}
Suppose that $P^\prime$ is an $(x_k,x_{2k+t-k+1})$-path on at least $k+2$ vertices in $G$.
By Claim \ref{AB}, $A$ is a component in $G-\{x_k,x_{2k+t-k+1}\}$ and can only be connected to $\{x_k,x_{2k+t-k+1}\}$.
If $A\cap V(P^\prime)\neq \emptyset$, then $V(P^\prime)\subseteq A$ and hence $|V(P^\prime)|\leq |A|+2= k+1$,
a contradiction. Hence $A\cap V(P^\prime)=\emptyset$.
Similarly we have $B\cap V(P^\prime)=\emptyset$.
Hence $V(P^\prime) \cap (A\cup B)=\emptyset$.
Then $x_kP^\prime x_{2k+t-k+1}Ax_k$ is a cycle on at least $2k+1$ vertices, a contradiction.
The proof is complete.
\end{proof}

The longest path starting from $x_k$ and ending at $x_{2k+t-k+1}$ in $G$ has at most $k+1$ vertices.
By Theorem \ref{Fan} and Proposition~\ref{claim9}, we have
$$e(G)\leq \frac{(k-2)(n-2)}{2}+2n-3<h(n,2k,k-1)\leq \max\{h(n,2k,k-1),h(n,2k+1,2)\},$$ a contraction to (\ref{bound for 2-connected}).
The proof of Case 1 is complete.

\medskip
\noindent{\bf Case 2.} $G^*$ contains $H_k(1)$ or $H_k(M_2)$ as a subgraph.
\medskip

By Lemma~\ref{property for H}(\romannumeral2), there is a path in $G^\ast$ on at
least $2k$ vertices starting from $x$ and ending at $y$ (clearly, this path does
not contain the adding edge $xy$, and hence is a path in $G$) in $H_k(1)$ or $H_k(M_2)$.
We may suppose that $N_{H}(x)\subseteq V(H_k(M_2))$ (or $V(H_k(1))$) and
$N_{H}(y)\subseteq V(H_k(M_2))$ (or $V(H_k(1))$), otherwise there is a
maximum path on at least $2k+1$ vertices starting and ending in $H$ and
we are also done by the previous proof of Case 1.
Since $\{x, y\}\not\subset C$ (otherwise $G$ contains $H_k(M_2)$ or $H_k(1)$ as a
subgraph), without loss of generality, we can assume that $x\in A\cup B\cup D$.

For $H_k(M_2)$, if $xy$ belongs to the embedded $M_2$, then since $|C|= k-1$ and
$N_H(x)\geq k$, $x$ is joint to a vertex of $(A\cup B\cup D)\setminus\{y\}$ in $G$;
if $xy$ does not belong to the embedded $M_2$, then $y\in C$ and as $N_H(x)\geq k$,
there is an edge in $G[A\cup B\cup D]$ that does not belong to $H_k(M_2)$.
In both of the above cases, by Lemma~\ref{property for H}(\romannumeral3), $G+xy$
contains a cycle on the $2k+1$ vertices, we are done by the previous proof of Case 1.
For $H_k(1)$, since all edges in $H_k(1)$ have a terminal vertex in $C$, we have $y\in C$.
Since $|C\setminus \{y\}|= k-1$ and $ N_H(x)\geq k$, $x$ is jointed to a vertex of $A\cup B\cup D$ in $G$.
By Lemma~\ref{property for H}(\romannumeral3) again, $G+xy$ contains a cycle on
$2k+1$ vertices, we are done by the previous proof of Case 1.
The proof of Case 2 is complete.

\medskip

\noindent{\bf Case 3.} $G^*$ contains $H_k(P_3)$ as a subgraph.

\medskip

Let $A=\{a_1,a_2,a_3\}$, $B$ and $C$ be the vertex sets of $H_k(P_3)$ (see its definition
at the beginning of this section).
We begin our proof of this case with the following claim.
\begin{claim4}\label{clai}
The vertices in $G^*-V(H_k(P_3))$ are independent and can only be adjacent to the
vertices in $B\cup \{a_2\}$. Moreover, the vertices in $G^*-A\cup B$ cannot be adjacent to both $a_2$ and some vertex in $B$.
\end{claim4}
\begin{proof}
If there is a vertex $p\in V(G^*)\setminus V(H_k(P_3))$ which is adjacent to $q\in (A\setminus \{a_2\})\cup C$,
then there is a path $P_0$ from $p$ to $q^\prime\in V(H_k(P_3))$ with $q^\prime\neq q$ (note that $G$ is 2-connected).
If $q^\prime q\in E(H_k(P_3))$; or $q^\prime=\{a_2\}$, $q\in C$; or $\{q^\prime,q\}=\{a_1,a_3\}$, by
Lemma \ref{H(2k+1)}(\romannumeral2), there is a path $P_1$ on the $2k$ vertices from $q$ to $q^\prime$ in $H_k(P_3)$.
Then $q^\prime P_1qpP_0q^\prime$ is a cycle of length at least $2k+1$ in $G^*$.
If $q^\prime,q\in C$; or $q^\prime\in A\setminus \{a_2\}$, $q\in C$; or $q\in A\setminus \{a_2\}$,
$q^\prime\in C$, by Lemma \ref{H(2k+1)}(\romannumeral1), we have $c(G^*)\geq 2k+2$.
In both of the above cases, there is a path in $G$ on at least $2k+1$ vertices that starts and ends
in $H$, and we are done by the previous proof of Case 1.
Hence, the vertices in $G^*-V(H_k(P_3))$ can only be adjacent to the vertices in $B\cup \{a_2\}$.

If $p\in G^*-V(H_k(P_3))$ is adjacent to $q\in B\cup \{a_2\}$ and not independent in
$G^*-V(H_k(P_3))$ (suppose that $p$ is adjacent to $p^\prime \in G^*-V(H_k(P_3))$),
then there is a path $P_0$ from $p^\prime$ to $q^\prime\in B\cup \{a_2\}$ with $q^\prime\neq q$ (note that $G$ is 2-connected).
By Lemma \ref{H(2k+1)}(\romannumeral4), there is a path $P_1$ on the $2k-1$ vertices from $q$ to $q^\prime$ in $H_k(P_3)$.
Then, $q^\prime P_1qpp^\prime P_0q^\prime$ is a cycle of length at least $2k+1$ in $G^*$.
Thus, we are done by the previous proof of Case 1.
Hence, the vertices in $G^*-V(H_k(P_3))$ are independent.

Assume that $z\in V(G)\setminus V(H_k(P_3))$ is adjacent to $a_2$ and $b_i\in B$,
since $G^*[B\cup C]$ contains a path $P_0$ on $2k-3$ vertices from $b_j\in B$ to
$b_i\in B$ ($i\neq j$), $a_1b_jP_0b_iza_2a_3$ is a path on $2k+1$ vertices.
Note that $a_1$ and $a_3$ are both adjacent to every vertex of $B\cup \{a_2\}$.
Thus, $G^*$ contains a copy of $H_k(1)$ and we are done in Case 2. The proof
of Claim~\ref{clai} is complete.
  \end{proof}

From  Claim~\ref{clai} and Lemma \ref{H(2k+1)}, we have
\begin{align*}
e(G^*)&=e(G^*[V(H_k(P_3))])+e(G^*[V(H_k(P_3)),V(G^*)-V(H_k(P_3))]) \\
&\leq h(2k+1,2k+1,k-1)+(n-2k-1)(k-1)     \\
&=h(n,2k+1,k-1).
\end{align*}
By (\ref{bound for 2-connected}), we have $e(G^*)>h(n,2k,k-1)+1=h(n,2k+1,k-1)-1$,
thus $e(G^*)=h(n,2k+1,k-1)$ and $G^*=H(n,2k+1,k-1)$ (each $N\in G^*-A\cup B$ is adjacent to all of $B\cup \{a_2\}$).
If $n=2k+1$, then $$e(G)=e(G^*)-1=h(n,2k+1,k-1)-1\leq h(2k+1,2k+1,2),$$
a contradiction to (\ref{bound for 2-connected}).
Hence $n\geq 2k+2$ and the subgraph induced by $A\cup B$ in $G^*=H(n,2k+1,k-1)$
is $K_{k+2}$ where there is at most one missing edge in $G$, thus we can easily find a $P_3$.
Moreover there is at most one missing edge between $B$ and $V(G)\setminus (A\cup B)$
in $G$ (note that $|V(G)\setminus (A\cup B)|\geq k$), which implies that $G$ contains
a copy of $H_k(P_3)$, a contradiction.
The proof of Case 3 is complete.

Hence we finish the proof of Claim~\ref{Hcomplete}.
\end{proof}

If $|V(H)|=\ell=2k$, then there is a cycle of length $2k+1$ (note that $G$ is 2-connected).
Thus by Claim~\ref{complete} we may assume $k+2\leq \ell\leq 2k-1$.
Apply to the graph $G$ the process of $\alpha$-disintegration with $\alpha=2k+1-\ell\leq k-1$. Let $H^\prime=H(G,2k+1-\ell)$.

Suppose that $H^\prime=H$. Since $k\ge 5$,
if $k+3\leq \ell\leq 2k-1$ or $2k+1\leq n\leq 2k+2$, we have $h(n,2k+1,2)>h(n,2k+1,k-2)$ and $h(n,2k+1,2)>h(n,2k,k-1)$. Then
\begin{align*}
e(G)&\leq {\ell \choose 2}+ (n-\ell)(2k+1-\ell)\\
 &=h(n, 2k+1,2k+1-l)\\
 &\leq\max\{h(n,2k+1,k-2),h(n,2k+1,2)\}\\
 &=h(n,2k+1,2)\\
 &=\max\{h(n,2k,k-1),h(n,2k+1,2)\},
\end{align*}
contradicting (\ref{bound for 2-connected}).

Thus, $|V(H)|=k+2$ and $n\geq 2k+3$.
By (\ref{bound for 2-connected}), we have
$$e(G)\geq h(n,2k,k-1)+1  = {k+2 \choose 2}+ (n-k-2)(k-1)-1.$$
Consider the $(k-1)$-disintegration of $G$, i.e. $H$.
We can delete at most one vertex such that the resulting graph
$F$ satisfies $\delta(F)=k-1$ and $\omega(F)=k+2$.
By Lemma \ref{Lemma:Yuan}, $c(F)\geq \min\{n-2,2k+1\}=2k+1$,
unless $F$ is $H(v(F),2k+1,k-1)$, implying that $F$ contains
$H_k(P_3)$ as a subgraph, which is a contradiction.

Therefore $H^\prime\neq H$, whence there exists a vertex $u \in H^\prime$ that is not joint to $v\in H$.
Thus, after adding $uv$, the obtained graph contains a cycle on at least $2k+1$ vertices,
or a copy of $H_k(1)$, $H_k(M_2)$, or $H_k(P_3)$.

If the obtained graph contains a cycle containing $uv$ on at least $2k+1$ vertices, there is
a path in $G$ on at least $2k+1$ vertices starting from $H$ and ending in $H^\prime$.
Let $xPy$ be a maximal such path with $x\in V(H)$ and $y\in V(H^\prime)$. (It is possible that $\{u,v\}\neq \{x,y\}$.)
Then, $d_P(x)\geq d_H(x)\geq \ell-1$ and $d_P(y)\geq d_{H^\prime}(y)\geq 2k+2-\ell$.
By Lemma~\ref{posa lemma}, we have $c(G)\geq 2k+1$, a contradiction.

If the obtained graph contains $H_k(1)$ ($H_k(M_2)$) as a subgraph, let $xy$ be the added edge.
By Lemma~\ref{property for H}(\romannumeral2), there is a path in $G$ on at least $2k$
vertices that start from $x$ in $H$ and end at $y$ in $H^\prime$.
We may suppose that $N_H(x)\subseteq V(H_k(1))$ (or $V(H_k(M_2))$) and $N_{H^\prime}(y)\subseteq V(H_k(1))$
(or $V(H_k(M_2))$),
otherwise there is a maximal such path with $x\in V(H)$ and $y\in V(H^\prime)$ on more than $2k$ vertices.
Then, similar to the last paragraph, by Lemma~\ref{posa lemma}, we have $c(G)\geq 2k+1$, a contradiction.
Furthermore, we have $V(H)\subseteq V(H_k(1))$ (or $V(H_k(M_2))$) by Claim \ref{Hcomplete}.
Since $|C|\leq k$ and $|H|=\ell\geq k+2$, there exist some $x^\prime \in (A\cup B\cup D)\cap H$
and $d_H(x^\prime)= \ell-1\geq k+1 >|C|$. Hence, there exists some $y^\prime \in (A\cup B\cup D)\cap H$
adjacent to $x^\prime \in (A\cup B\cup D)\cap H$. Thus, by Lemma~\ref{property for H}(\romannumeral3),
$G+xy$ contains a cycle of length $2k+1$.
This implies that there is a path on $2k+1$ vertices starting from $H$ and ending at $H^\prime$,
a contradiction by Lemma~\ref{posa lemma}.

If the obtained graph contains $H_k(P_3)$ as a subgraph, let $xy$ be the added edge. According to
Lemma~\ref{H(2k+1)}(\romannumeral2), there is a path in $G$ on at least $2k$ vertices that starts
from $x$ in $H$ and ends at $y$ in $H^\prime$, unless $xy$ is the edge between $a_2$ and $B$.
If $xy$ is the edge between $a_2$ and $B$, then $d_{H_k(P_3)}(a_1)=d_{H_k(P_3)}(a_2)=d_{H_k(P_3)}(a_3)=k$,
we have $\{a_1,a_2,a_3\}\subseteq H$. By Claim \ref{Hcomplete}, $a_1$ is joint to $a_3$.
Then by Lemma~\ref{H(2k+1)}(\romannumeral4), there is still a path in $G$ on at least $2k$ vertices
starting from $a_2$ in $H$ and ending in $B$.
Similarly as in the last paragraph, we may suppose that $N_H(x)\subseteq V(H_k(P_3))$, $N_{H^\prime}(y)\subseteq V(H_k(P_3))$
and $V(H)\subseteq V(H_k(P_3))$ by Claim \ref{Hcomplete}.
Since $|B\cup \{a_2\}|=k$ and $|H|=\ell\geq k+2$, there exist some $x^\prime \in (\{a_1,a_3\}\cup C)\cap H$
and $d_H(x^\prime)= \ell-1\geq k+1 >|B\cup\{a_2\}|$. Hence, there exists some $y^\prime \in (\{a_1,a_3\}\cup C)\cap H$
adjacent to $x^\prime \in (\{a_1,a_3\}\cup C)\cap H$.
By Lemma~\ref{H(2k+1)}, $G+xy$ contains a cycle of length $2k+1$.
Therefore, there is a path on $2k+1$ vertices starting from $H$ and ending at $H^\prime$,
a contradiction by Lemma~\ref{posa lemma}.

The proof of Lemma \ref{Lemma:Stablity-Kopylov2} is complete. \hfill $\square$

\end{document}